\documentclass[reqno,a4paper,12pt]{amsart}

\usepackage[all,poly]{xy}
\usepackage{amsfonts}
\usepackage[mathcal]{eucal}
\usepackage{eufrak}
\usepackage{amssymb}
\usepackage{amsmath}
\usepackage{mathrsfs}
\usepackage{color}
\usepackage[colorlinks]{hyperref}
\usepackage{enumerate}

\definecolor{citecol}{RGB}{145, 1, 1}

\hypersetup{colorlinks=true,citecolor=citecol,linkcolor=blue,linktocpage=true}

\usepackage{comment}

\theoremstyle{plain}
\newtheorem {lemma}{Lemma}[section] 
\newtheorem {theorem}[lemma]{Theorem}

\newtheorem {corollary}[lemma]{Corollary}

\newtheorem {proposition}[lemma]{Proposition}

\theoremstyle{definition}
\newtheorem {remark}[lemma]{Remark}

\newtheorem {example}[lemma]{Example}

\theoremstyle{definition}

\newtheorem{deff}[lemma]{Definition}{}

\newcommand{\op}{\operatorname{op}}

\newcommand{\V}{\operatorname{\mathcal V}}
\newcommand{\rr}{\operatorname{\mathbf r}}
\newcommand{\reg}{\operatorname{reg}}
\newcommand{\ssink}{\operatorname{sink}}

\newcommand{\K}{\mathsf{k}}

\newcommand{\coker}{\operatorname{coker}}

\newcommand{\SP}{\operatorname{SP}}

\begin{document}

\title[Sandpile models and Leavitt path algebras]{Connections between Abelian sandpile models and\\ the $K$-theory of weighted Leavitt path algebras}


\author{Gene Abrams}
\address{Department of Mathematics\\University of Colorado\\Colorado Springs, CO 80918, USA}
\email{abrams@math.uccs.edu}

\author{Roozbeh Hazrat}\address{
Centre for Research in Mathematics and Data Science\\
Western Sydney University\\
Australia}
\email{r.hazrat@westernsydney.edu.au}

\subjclass[2010]{16S88, 05C57} 
\keywords{Abelian sandpile model, sandpile monoid, weighted Leavitt path algebra, Grothendieck group}


\begin{abstract}
In our main result, we establish that any conical sandpile monoid $M = \SP(E)$ of a directed sandpile graph $E$ can be realised as the $\mathcal{V}$-monoid of a weighted Leavitt path algebra $L_\K(F,w)$ (where $F$ is an explicitly constructed subgraph of $E$), and consequently, the sandpile group $\mathcal{G}(E)$ is realised as the Grothendieck group $K_0(L_\K(F,w))$.  
 Additionally, we describe the conical sandpile monoids which arise as the $\mathcal{V}$-monoid of a standard (i.e., unweighted) Leavitt path algebra.  
\end{abstract}


\maketitle

\section{Introduction} \label{introkai}

``\emph{Your descendants will be like the dust of the earth, and you will spread out to the west and to the east, to the north and to the south.}" Genesis 28:14. 

The notion of sandpile models encapsulates how objects spread and evolve along a grid. The models were conceived in 1987 in the seminal paper~\cite{bak} by Bak, Tang and Wiesenfeld  as  examples of \emph{self-organized criticality}, or the tendency of physical systems to organise themselves without any input from outside the system, toward critical but barely stable states. 
The models have been used to describe phenomena such as forest fires, traffic jams, stock market fluctuations, etc. The book of Bak~\cite{bakbook} describes how   events in nature apparently follow this type of behaviour. 

 The mathematical formulation of the model is as follows. Consider a sandpile graph, namely a (finite, directed) graph $E$ with a distinguished sink  vertex $s$ such that there is a (directed) path from any vertex of $E$ to $s$. Consider a collection of grains of sand placed on each vertex of the sandpile graph (a {\it configuration}). A vertex is {\it unstable} if it has the same or more grains of sand than the number of edges emitting from it.  In this case the vertex {\it topples}  by sending one grain along each edge emitting from the vertex to each neighbouring vertex.  This toppling may cause neighbouring vertices to become unstable. The assumption is that grains arriving to the vertex $s$ vanish. A configuration is {\it stable} if no vertex is unstable.    Two foundational results used in the theory are: the order of toppling does not matter
(thus ``abelian'' sandpile models); and any configuration can be sequentially toppled
 to reach a unique stable configuration. This allows for the construction of  a finite monoid associated to an abelian sandpile model. Let $\SP(E)$ be the set of all stable configurations of a sandpile graph $E$. The ``sum"  of two configurations is interpreted as simply adding the number of grains of sand at each vertex corresponding to the two configurations.
The operation of addition followed by stabilisation endows the set $\SP(E)$  with the structure of a commutative monoid, called the \emph{sandpile monoid}  (see Definition~\ref{sandpilemonoiddef} for a precise definition). 
The smallest set of configurations of $\SP(E)$ which is closed under adding any grains to them (i.e., the recurrent configurations) forms a group, called the \emph{sandpile group} $\mathcal{G}(E)$ associated to $E$. These algebraic structures constitute one of the main themes of the subject.   In a major work~\cite{D1}, D.~Dhar championed the use of $\mathcal{G}(E)$ as an invariant which proved to capture many properties of the model.   A more algebraic study of these monoids and their groups is carried out in \cite{BT,CGGMS}. The books of Klivans~\cite{Kl} and Corry and Perkinson~\cite{corper} give self-contained treatments of the subject of sandpile models.

In a different realm, the notion of Leavitt path algebras $L_\K(E)$ associated to directed graphs $E$, with coefficients in a field $\K$,  were introduced in 2005~\cite{aap05,amp}. These are a generalisation of algebras (denoted by  $L_\K(1,1+k)$)  introduced by William Leavitt in  1962~\cite{vitt62} as a ``universal'' $\K$-algebra  $A$ of type $(1,1+k)$, so that  $A\cong A^{1+k}$ as right $A$-modules, where $k\in \mathbb N^+$. The study of the commutative monoid $\mathcal V(B)$ of isomorphism classes of finitely generated projective right modules over a unital ring $B$ (with operation $\oplus$)  goes back to the work of Grothendieck and Serre~\cite{magurn}. For a Leavitt path algebra $L_\K(E)$, the monoid $\mathcal V(L_\K(E))$ has received substantial attention due to two separate impetuses. On the one hand, these monoids are conical and refinement, and thus related to the realisation problem for von Neumann regular rings~\cite{goodearl}. (The problem asks whether every countable conical refinement monoid can be realised as the monoid of a von Neumann regular ring~\cite[\S 7.2.3]{TheBook}.)  On the other hand,  the group completion of $\mathcal V(L_\K(E))$ is the Grothendieck group $K_0(L_\K(E))$. This group has been used effectively as a complete invariant for a large class of graph $C^*$-algebras (these are the analytic versions of Leavitt path algebras). Thus there is a hope that this group could play a similar role in the algebraic setting as well. Whether $K_0$ classifies these algebras has remained one of the most important and yet elusive questions in the theory of Leavitt path algebras~\cite[\S 7.3.1]{TheBook}.  We will touch on this question in Section \ref{connectionsection}. 

In fact Leavitt established much more in \cite{vitt62} than mentioned in the previous paragraph; indeed, he showed that, for any $n,k \in \mathbb{N}$,  there is a universal $\K$-algebra $A$ of type $(n,n+k)$ (denoted $L_\K(n,n+k)$)  for which $A^n\cong A^{n+k}$ as right $A$-modules.  When $n\geq 2$, this universal algebra is {\it not} realizable as a Leavitt path algebra.   With this in mind, the notion of {\it weighted}  Leavitt path algebras $L_\K(E,w)$ associated to weighted graphs $(E,w)$ were introduced by the second author in 2011 (\cite{H}; see \cite{Pre} for a nice overview of this topic).  The weighted Leavitt path algebras $L_\K(E,w)$  provide a natural context in which all of Leavitt's algebras (corresponding to any pair $n,k \in \mathbb{N}$) can be realised as a specific example.
As well, the monoid $\mathcal V(L_\K(E,w))$ and the corresponding Grothendieck group $K_0(L_\K(E,w))$ have been completely described in the two works \cite{H} and \cite{P}.  

In this article we tie the notions of sandpile models and weighted Leavitt path algebras together.   We start in Section \ref{weightedgraphssec} by defining {\it weighted graph monoids}, and show that these include the sandpile monoids.   We then show in Section \ref{confluencesec}  that any weighted graph monoid satisfies the {\it Confluence Property}.  In particular, this property will allow us to identify the unit group of the weighted graph monoid (Proposition \ref{cansurv}).     With these preliminaries established, we move towards our main result 
by casting the theory of sandpile monoids in an abstract algebraic setting.  To do so we utilise two foundational works of George Bergman:  the Diamond Lemma~\cite{bergman78}, and the Universal Ring Construction~\cite{B}. The first work will allow us to re-establish that sandpile  configurations have unique normal forms
(Section \ref{diamondsec}),   while  the second will allow us to realise monoids of certain algebras as graph monoids 
(Section \ref{wLpasection}, see especially Theorem \ref{kspcts}).  In Section  \ref{connectionsection} we achieve our main result  (Theorem \ref{conicrit}).   The key consequence of Theorem \ref{conicrit} is that,  starting from a sandpile graph $E$ having a specified form ('conical'), and considering $E$ as a  weighted graph $(E,w)$ for an appropriate weight $w$, and considering a specific weighted subgraph $(F,w_r)$ of $E$, we obtain that  $\SP(E) \cong \mathcal V(L_\K(F,w_r))$.  Consequently, we conclude that $\mathcal{G}(E) \cong K_0(L_\K(F,w_r))$ as well.    
Furthermore, we identify those sandpile monoids which arise as the $\mathcal{V}$-monoid of a standard (i.e., unweighted)  Leavitt path algebra.  

The connection established in Theorem \ref{conicrit}   between sandpile monoids and weighted Leavitt path algebras  allows us to naturally associate an algebra, a \emph{sandpile algebra},  to the theory of sandpile models, thereby 
opening up  an avenue by which to investigate sandpile models via the structure of the sandpile algebras, and vice versa.   



\section{Weighted graph monoids and sandpile monoids: \ basic properties} \label{weightedgraphssec}


Throughout we write $\mathbb N$ for the set of non-negative integers, and $\mathbb N^+$ for the set of positive integers.

We start by setting some terminology and notation about graphs.  A \emph{directed graph} is a quadruple $E=(E^0,E^1,s,r)$, where $E^0$ and $E^1$ are sets and $s,r:E^1\rightarrow E^0$ are maps.   Throughout, ``graph" will always mean ``directed graph".  The elements of $E^0$ are called 
\emph{vertices} and the elements of $E^1$ \emph{edges}.   (We allow the empty set to be viewed as a graph with $E^0 = E^1 = \emptyset.$)   If $e$ is an edge, then $s(e)$ is called its \emph{source} and $r(e)$ its \emph{range}. If $v$ is a vertex and $e$ an edge, we say that $v$ {\it emits} $e$ if $s(e)=v$, and $v$ {\it receives} $e$ if $r(e)=v$.   An edge $e$ is called a {\it loop at} $v$ in case $s(e) = v = r(e)$.  
A vertex is called a
 {\it sink} if it emits no edges, and is called \emph{irrelevant} in case it emits exactly one edge.  A graph is called {\it reduced} if it contains no irrelevant vertices.  
A vertex is called \emph{regular} if it is not a sink and does not emit infinitely many edges.  The subset of $E^0$ consisting of all the regular vertices is denoted by $E^0_{\reg}$. Similarly, the subset of $E^0$ consisting of all the sinks  is denoted by $E^0_{\ssink}$.

The \emph{out-degree} of a vertex $v$ is defined as $|s^{-1}(v)|$. A graph is called \emph{row-finite} if any vertex emits a finite number (possibly zero) of edges. The graph $E$ is called  \emph{finite} if $E^0$ and $E^1$ are finite sets.   In this paper we assume throughout that all graphs are  row-finite,  and thus $E^0_{\reg}$ consists of all vertices which are not sinks.

A  {\it path}  $p$ in $E$ is
a sequence $p=e_{1}e_{2}\cdots e_{n}$ of edges in $E$ such that
$r(e_{i})=s(e_{i+1})$ for $1\leq i\leq n-1$. We define $s(p) = s(e_{1})$, and $r(p) =r(e_{n})$. 
 By definition, the \emph{length} $|p|$ of $p$ is $n$. We assign the length zero to vertices.  A \emph{closed path} (based at $v$) is a  path $p$ such that $s(p)=r(p)=v$. A \emph{cycle} (based at $v$) is a closed path $p=e_1 e_2 \cdots e_n$ based at $v$ such that $s(e_i)\neq s(e_j)$ for any $i\neq j$. 
 
A collection of graphs which will be the primary focus of this article is now easy to define.   
 
 \begin{deff}\label{sanddefmon}
A finite directed graph $E$ is called a \emph{sandpile graph} if $E$ has a unique sink (denote it by $s$), and for every $v\in E^0$ there is a  path $p$ with $s(p)=v$ and $r(p) = s$.  
\end{deff}
 
 A subset $H \subseteq E^0$ is said to be \emph{hereditary} if
for any $e \in E^1$,  $s(e)\in H$ implies $r(e)\in H$. 
 For a hereditary subset $H$ of $E$, we define the quotient graph $E/H$ as follows:
  $$(E/ H)^0=E^0\setminus H^0   \ \ \mbox{and} \ \  (E/H)^1=\{e\in E^1\;| \ r(e)\notin H\}.$$ The source and range maps of $E/H$ are the source and range maps  restricted from the graph $E$.
 We have a natural (inclusion) morphism $\phi:E/H\rightarrow E$ which is injective on vertices and edges.


\begin{deff}\label{weightedgraphdef}
A \emph{weighted graph} is a pair $(E,w)$, where $E$ is a graph and $w:E^1\rightarrow \mathbb N^+$ is a map. If $e\in E^1$, then $w(e)$ is called the \emph{weight} of $e$. 

For each regular vertex $v$ in a weighted graph $(E,w)$ we set $w(v):=\max\{w(e)\mid e\in s^{-1}(v)\}$.  This gives  a map (called $w$ again) $w:E_{\reg}^0\rightarrow \mathbb N^+$. 

A weighted graph is called a \emph{vertex weighted graph} if for each regular vertex $v$ in the graph, the weights of all edges emitting from $v$  are equal one to the other;  i.e., if $w(e) = w(e')$ for all $e,e' \in s^{-1}(v)$.  In this case $w(v)$
coincides with the weight of any edge emitting from $v$.     

A vertex weighted graph is called a \emph{balanced weighted graph} if  $w(v)=|s^{-1}(v)|$ for all $v\in E_{\reg}^0$.   
\end{deff}


\begin{remark}\label{weightonquotientremark}
 Note that any (row-finite)  directed graph can be given the structure of a balanced weighted graph by assigning $w(v)=|s^{-1}(v)|$, for all $v\in E_{\reg}^0$.   (By definition we are not required to assign a weight to any sinks in $E$.)  

In particular, suppose that $H$ is a hereditary 
subset of the weighted graph $(E,w)$.   We form the quotient graph $E/H$.   Then there are two natural ways to define a weight function on $E/H$:    restrict the weight function on $E$ to $E/H$, or impose the structure of a balanced weighted graph on $E/H$.  We note that if $w$ is a balanced weight function on $E$, then  the restriction function of $w$ to $E/H$ is {\it not} in general the same as  the balanced weight function on $E/H$.    In situations where additional clarity is warranted, we will write $(E/H, w_r)$ to indicate that the weight function being considered on $E/H$ is the restriction function of the weight function $w$ to $E/H$.  
\end{remark}

\smallskip

  We continue by setting some terminology and notation about monoids.   Let $(M,+)$ be a commutative  monoid.  We define the \emph{algebraic  pre-ordering} on  $M$ by setting $a\leq b$ if $b=a+c$, for some $c\in M$. 
 A commutative monoid $M$ is called \emph{conical} if $a+b=0$ implies $a=b=0$, and  is called \emph{cancellative} if $a+b=a+c$ implies $b=c$ for all $a,b,c\in M$. The monoid $M$ is called \emph{refinement} if, whenever $a,b,c,d \in M$ have $a+b=c+d$, then there are $e_1,e_2,e_3,e_4\in M$ such that $a=e_1+e_2$, $b=e_3+e_4$ and $c=e_1+e_3$, $d=e_2+e_4$. 
An element $0 \not = a\in M$ is called an \emph{atom} if whenever $a=b+c$ then $b=0$ or $c=0$. 
We say a commutative monoid is \emph{atom-cancellative} if for any atom $a\in M$, the equality $a+m=a+m'$ implies $m=m'$.

Let $Y$ be a submonoid of $M$ (so $Y$ is closed under $+$, and contains $0$). For $a, b\in M$,  write $a\sim_{Y} b$ if there exist $i,j\in Y$ such that  $a+i=b+j$ in $M$. This is a congruence relation  and thus one can form the quotient monoid   $M/\sim$ which we will denote by $M/Y$. 
 The congruence $a\sim_{Y} b$ is equivalent to $(a+Y) \cap (b+Y) \not = \emptyset$. Observe that $a\sim_{Y} 0$ in $M$ for any $a\in Y$.

A subset $I$ of $M$ is called an \emph{ideal} of $M$ if $m+I\subseteq I$, for any $m\in M$.

For a commutative monoid $M$, the set 
$$Z(M)=\{a\in M \mid a+b=0, \text{ for some } b\in M\}$$
 is an abelian group, called the \emph{unit group} of $M$.  Clearly $Z(M)=0$ if and only if $M$ is conical and $Z(M)=M$ if and only if $M$ is a group. We refer the reader to \cite{RGS} for the general theory of commutative monoids and \cite{wehrung} for  refinement commutative monoids. 

The following simple lemma will be quite useful later, especially in the proof of 
Proposition~\ref{refinement}. 

\begin{lemma}\label{atomcancel}
Any refinement monoid is atom-cancellative. Furthermore, any monoid which is finite, conical and refinement  has no atoms. 
\end{lemma}
\begin{proof}
Let $a,m,m' \in M$, where $M$ is a refinement monoid and $a$ is an atom such that $a+m=a+m'$. Then there are $e_1,e_2,e_3,e_4\in M$ such that $a=e_1+e_2$, $m=e_3+e_4$ and $a=e_1+e_3$, $m'=e_2+e_4$. Since $a$ is an atom, either $e_1=0$ or $e_2=0$. If $e_1=0$ then $e_2=e_3=a$ and thus $m=a+e_4=m'$. If $e_2=0$ then $e_1=a$ and $e_3=0$ which gives that $m=e_4=m'$. 

Now let $M$ be  finite, conical  and refinement. Suppose $a\in M$ is an atom. Consider the translation map $\phi_a: M\backslash \{0\} \rightarrow M\backslash \{0\}, m\mapsto a+m$. Since $a$ is an atom, the first part of the lemma implies $\phi_a$ is injective. Thus $\phi_a$ is surjective as  $M\backslash \{0\}$ is finite. But $a$ can't be in the image of $\phi_a$, a contradiction. 
\end{proof}

Let $M$ be a commutative monoid. The \emph{group completion} of $M$, denoted $\mathcal{G}(M)$,  is the free abelian group generated by elements of the set $\{[a]\mid  a\in M\}$ subject to the relations $[a+b]-[a]-[b]$, for any $a,b\in M$. The homomorphism $M\rightarrow \mathcal{G}(M); m\mapsto [m]$ is initial among all such maps~\cite{magurn}. The {\it Grothendieck group} of a unital ring $A$, denoted $K_0(A)$, is by definition  the group completion of the monoid $\mathcal V(A)$. When the monoid $M$ is finite, the group completion $\mathcal{G}(M)$ is in fact a subgroup of $M$.

\begin{lemma}\label{grothenlem}
Let $M$ be a finite commutative monoid. Then the group completion $\mathcal{G}(M)$  of $M$ is isomorphic to the smallest ideal of $M$. 
\end{lemma}

\begin{proof}

Let $I$ denote $\bigcap_{a\in M}(a+M)$.   Since $M$ is finite, $t:= \sum_{a\in M} a$ is a well-defined element of $ M$, and clearly $t \in I$, so that $I \neq \emptyset$.  It is easy to show that $I$ is an ideal of $M$.      Let $J$ be any ideal of $M$.  For each $a\in J$ and $m\in M$ we have $m+a \in J$, so $a+M \subseteq  J$, so $I\subseteq J$.   So $I$ is the smallest ideal of $M$.   

We show that $I$ is a group.  For each $x\in I$, the subset $I+x := \{\iota + x \ | \ \iota \in I\}$ is an ideal of $M$ contained in $I$, so $I+x = I$.   In particular there exists $z\in I$ with $z+x = x$.   As well, for each $y\in I$ there is $u \in I$ with $y = u+x$.  Then $y+z = u + x + z =  u + z + x = u + x = y$, so $z$ is an identity element for $I$.   But $I + x = I$ also gives that there exists $x' \in I$ with $x' + x = z$.   So $I$ is a group.

Define $\phi: M \to I$ by setting $\phi(m) = m+z$ for each $m\in M$.   (Note that $m+z \in I$ as $z\in I$ and $I$ is an ideal of $M$.)   Since $z+z = z$, $\phi$ is a monoid homomorphism.  Let $\Gamma$ be any abelian group, and let   $\psi:M \rightarrow \Gamma$ be a monoid homomorphism.   
Then easily $\psi = \psi |_I  \circ \phi  $.    
Thus  $I\cong \mathcal{G}(M)$.
\end{proof}


\begin{example}\label{Mplusexample}
We clarify some of the previous observations in the context of a specific example, one to which we will make reference often in the sequel.   For $n,k \in \mathbb{N}^+$, let $M_{n,n+k}$ denote the finite commutative monoid
$$M_{n,n+k} := \{ 0, x, 2x, \dots , nx , \dots, (n+k-1)x \} , \ \  \ \ \mbox{with relation} \ (n+k)x = nx.$$
Then $M_{n,n+k}$ is clearly conical, and not cancellative.  The atoms of $M_{n,n+k}$ consist of the elements  $x, 2x, \dots, (n-1)x$.  (In particular $M_{1,1+k}$ has no atoms.)  The monoid  $M_{n,n+k}$ is refinement if and only if $n=1$.  (To show that  $M_{n,n+k}$ is not refinement for $n\geq 2$, consider the equation $x + (n-1)x = x + (n+k-1)x$ in $M_{n,n+k}$.)

Finally, the      smallest ideal of $M_{n,n+k}$     consists of the subset $\{ nx, \dots , (n+k-1)x\}$.    By Lemma \ref{grothenlem} this is also then the group completion $\mathcal{G}(M_{n,n+k})$ of $M_{n,n+k}$ .   So we see that  $\mathcal{G}(M_{n,n+k})$ is isomorphic to the cyclic group of order $k$:  it  is easy to show that the identity element of this group is  $tx$, where $t$ is the unique multiple of $k$ appearing in the integer interval  $[n, n+k-1]$, and that $(t+1)x$ is a generator for $\mathcal{G}(M_{n,n+k})$.   
 \end{example}


We complete this section by tying together the two previously presented themes, by associating 
certain monoids to  weighted graphs.   
Once the specific description and basic properties of these graph monoids have been established, in the subsequent two sections we'll look at the Confluence  and the Reduction-Uniqueness properties, respectively.


\begin{deff}\label{genmon}
Let $(E,w)$ be a  weighted graph.
 We assign $w(v)=1$ if $v$ is a sink. The \emph{reduced graph monoid} $M(E,w)$ associated to $(E,w)$ is defined to be 
  \begin{equation}\label{monvertex}
M(E,w) \  \ := \  \ \mathbb F_E \  \Big /  \ \Big \langle w(v)  v= \sum_{e \in s^{-1}(v)} r(e) \, \big | \, v\in E^0\Big \rangle.
\end{equation}
Here $\mathbb F_E$ is the free commutative monoid on the set $E^0$ of vertices of $E$.  (If $E^0$ is the empty set then we interpret $\mathbb F_E$ as the zero monoid.)

Note that if $s$ is a sink, then since the summation in~(\ref{monvertex}) is over an empty set, the relation  corresponding to $s$ in $M(E,w)$ reduces to:  \  $s=0$.    (This is why we use the terminology {\it reduced} graph monoid to describe $M(E,w)$.)  
\end{deff}

\begin{example}\label{weightedgraphmonoidexample}  Let $n,k \in \mathbb{N}^+$.   Let $E_{n,n+k}$ denote the vertex weighted graph having one vertex $v$, with $n$ loops at $v$, each having weight $n+k$.  Pictorially, $E_{n,n+k}$ can be viewed as 
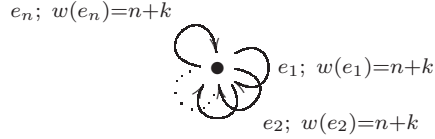
\begin{equation*}
\xymatrix{
& \bullet \ar@{.}@(l,d) \ar@(ur,dr)^{e_1; \ w(e_1)=n+k} \ar@(r,d)^{e_2; \ w(e_2)=n+k} \ar@(dr,dl)^{} \ar@(l,u)^{e_n; \ w(e_n)=n+k}& 
}
\end{equation*}
So $w(v) = n+k$, and 
$$M(E_{n,n+k}) :=   \mathbb{F}_{v} \  \Big / \ \langle (n+k)v = nv \rangle \  \cong \ M_{n,n+k},$$
where $M_{n,n+k}$ is the monoid defined in Example \ref{Mplusexample}. 
\end{example}

\begin{remark}\label{infiniteMremark}
We note that even if $E$ is finite, certainly $M(E,w)$ might be infinite.  As an easy example, if $E$ is the graph with one vertex $v$ and one loop $e$ based at $v$ having $w(e)=1$, then $w(v) = 1$ and  \ $M(E,w) = \mathbb{F}_{v} / \langle 1v = v\rangle \cong \mathbb{N}$. 
\end{remark}

There is an explicit description of the congruence on $\mathbb F_E$ given by the defining relations of $M(E,w)$ in (\ref{monvertex}), as follows. For $v\in E^0$, with  weight $w(v)$, define the ``$\rr$-transform on $\mathbb{F}_E$" by setting  
\begin{equation}\label{shtrans}
\rr(w(v) v) \ : = \ \sum_{e\in s^{-1}(v)}r(e).
\end{equation}
 The nonzero elements of $\mathbb F_E$ can be written uniquely up to permutation as $\sum_{i=1}^{n}k_iv_{i}$, where $v_{i}$ are distinct vertices and $k_i\in \mathbb N^+$. Define a binary relation
$\rightarrow_{1}$ on $\mathbb F_E$ by 
\begin{equation}\label{hfgtrgt655}
\sum_{i=1}^{n}k_iv_{i} \ \longrightarrow_{1} \  \Big( \sum_{i\neq
j}k_iv_{i} \Big) +(k_j-w(v_j))v_j+ \rr(w(v_j)v_j),
\end{equation}
whenever $j\in \{1, \cdots, n\}$ and 
$k_j\geq w(v_{j})$.   (Clearly the transformation described by $\rightarrow_1$ models the toppling process at $v_j$  in the sandpile model.)    Let $\rightarrow$ be the transitive and reflexive closure of $\rightarrow_{1}$
on $\mathbb F_E$.  Namely 
\begin{equation}\label{hfgtrgt6551}
a\rightarrow b    \ \  \ \text{ if }  a=b, \ \mbox{or}  \ a=a_0 \rightarrow_1 a_1 \rightarrow_1 \dots \rightarrow_1 a_k=b.
\end{equation}

\noindent
Finally, let   $\sim$   be the congruence on $\mathbb F_E$ generated by the relation $\rightarrow$.   That is, $a\sim b$ in case  there is a string $a=a_0, a_1,\dots, a_n=b$ in $\mathbb F_E$ such that $a_i\rightarrow_1 a_{i+1}$ or $a_{i+1}\rightarrow_1 a_{i}$ for each $0\leq i \leq n-1$.   Then 
$$M(E,w)=\mathbb F_E/\sim.$$
\noindent
To avoid cumbersome equivalence class notation,  is standard (but not technically correct) to denote the elements of $M(E,w)$ and the elements of $\mathbb F_E$ using the same symbols.  For instance, we  will sometimes write $a=b$ in $M(E,w)$ for  elements $a,b \in \mathbb F_E$.

\begin{deff}\label{sandpilemonoiddef}
For a sandpile graph $E$, the {\it sandpile monoid} $\SP(E)$    has been described informally in the Introduction; we revisit that description, and provide here a formal definition.   
(See e.g. \cite{BT} or \cite{CGGMS} for additional information).  For $k_j \in \mathbb{N}$ we envision  $k_j$ grains of sand sitting on each non-sink vertex $v_j$ of $E$. 
When the number of grains of sand sitting at $v_j$  is equal to or larger than the number of edges emitting from $v_j$, then   $v_j$  fires one grain of sand along each edge emitted by $v_j$  to its  adjacent vertices.     Further, any grains which arrive at the sink $s$ are understood to vanish.   

More formally, $\SP(E)$ is the monoid defined by generators and relations as
\begin{equation}\label{SP(E)def}
\SP(E) \ \ :=  \ \ \mathbb F_E \ \Big / \ \Big \langle  s = 0;  \     |s^{-1}(v)| v= \sum_{e \in s^{-1}(v)} r(e), \ \mbox{for} \  v\in E^0\ \setminus\{s\} \Big \rangle.
\end{equation}
Furthermore, the {\it sandpile group} $\mathcal{G}(E)$ is defined by setting
$$\mathcal{G}(E) := \mbox{ the smallest ideal of }  \SP(E).$$

\noindent
\end{deff}

\begin{example}\label{Mnn+kissandpile}
Let $G$ be the graph pictured here.    
$$
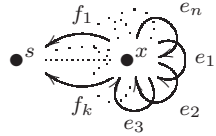
\xymatrix{
\!\!\!    \bullet^s   &   \bullet^x \ar@{.}[l]  \ar@{.}@/_{5pt}/ [l]  \ar@/_{10pt}/ [l]_{f_1}  \ar@/^{10pt}/ [l]^{f_k}     \ar@{.}@(l,d) \ar@(ur,dr)^{e_{1}} \ar@(r,d)^{e_{2}} \ar@(dr,dl)^{e_{3}} 
\ar@{.}@(l,u) \ar@(u,r)^{e_n}  \ar@{.}@(ul,ur)& 
}$$
So $G$ has one sink $s$, and one non-sink vertex $x$, with $n$ loops based at $x$, and $k$ edges from $x$ to $s$.    Thus $G$ is a sandpile graph.   Moreover,  
$$\SP(G) = \mathbb{F}_{x,s} \Big / \langle s = 0; \  (n+k)x = nx + ks \rangle \  \cong M_{n,n+k},$$
 where $M_{n,n+k}$ is the monoid defined in Example \ref{Mplusexample}.  
\end{example}

\begin{remark}\label{sandpileex} 
 {\bf  (A key connection) }    Let $E$ be a sandpile graph.   
Considering $E$ as a balanced weighted graph (so  $w(v)=|s^{-1}(v)|$ for each non-sink vertex $v$ of $E$),   and defining $w(s)=1$ for the sink $s$,   Definitions  \ref{genmon}   and  \ref{sandpilemonoiddef} immediately give 
$$\SP(E) = M(E,w).$$ 
So every sandpile monoid is the monoid of a (balanced) weighted graph.  The converse is certainly not true: as one justification, we will show below that the monoid $\SP(E)$ is finite for any finite graph $E$, while $M(E,w)$ need not be finite in general (see Remark \ref{infiniteMremark}).  

Of course it is not coincidental that the monoid $M(E,w)$ of Example \ref{weightedgraphmonoidexample} is the same as the monoid $\SP(G)$ of Example \ref{Mnn+kissandpile}.  We will make the connection precise in Theorem \ref{conicrit}.   
\end{remark}




\begin{remark}\label{sandpilemonoidremark}
Let $E$ be a sandpile graph.  Suppose $v$ is an irrelevant vertex in $E$, and denote by $w$ the vertex $r(e)$, where $e$ is the unique edge having $s(e) = v$.   By ``collapsing" $E$ at $v$ we mean eliminating $v$ and $e$ from $E$, and, for any edge $f$ in $E$ having $r(f) = v$, we assign $r(f) = w$.   Let $G$ be the graph gotten from $E$ by continuing to collapse at each vertex until no irrelevant vertices remain.    Then clearly $G$ is a reduced sandpile graph, and $\SP(E) \cong \SP(G)$.  In particular, when studying the structure of  $\SP(E)$ we may assume without loss that $E$ is reduced.  
\end{remark}

\section{The Confluence Property}\label{confluencesec}

In this section we show that the reduced graph monoids $M(E,w)$ have a certain ``confluence"  property. Informally, this means that whenever two elements $a,b$ are equal when viewed in the quotient monoid $M(E,w)=\mathbb F_E/\sim$, then there is a common element in  $\mathbb F_E$ to which both $a$ and $b$ flow.      The proof of the Confluence Lemma 
we present here is similar to the proof of the corresponding result in the case of the graph monoid $M_E$ for an unweighted graph $E$ ~\cite[\S 6.3]{amp}, but with the added complexity of having weights and sinks in the relations.  

\begin{lemma}[The Confluence Lemma]\label{aralem6}
Let $(E,w)$ be a 
weighted graph and $M(E,w)$ its associated monoid. For $a, b \in \mathbb F_E$, we have   $a=b$ in $M(E,w)$  (i.e., $a\sim b$ in $\mathbb F_E$)  if and only if there exists  $c \in \mathbb F_E$ such that  $a \rightarrow c$ and $b\rightarrow c$. 

In particular, any sandpile monoid satisfies this ``Confluence Property".  
\end{lemma}

\begin{proof}
The sufficiency follows directly from the definition of $\sim$, as it is  the symmetric and transitive closure of the relation $\rightarrow_1$.

To establish the necessity, we start by observing the following   fact. Let $a$ be an element of $\mathbb F_E\backslash \{0\}$ such that $kv$ appears in its presentation, where $v$ is a vertex having  weight $k$. Writing $a=kv+a'$ then by definition (\ref{hfgtrgt655}) a transformation $\rightarrow_1$  can take either the form $a\rightarrow_1 \rr(kv)+a'$ or $a\rightarrow_1 kv+a''$, where $a'\rightarrow_1a''$. 

Now suppose $a=b$ in $M_{(E,w)}$. Then there is a string $a=a_0, a_1,\dots, a_n=b$ in $\mathbb F_E$ such that $a_i\rightarrow_1 a_{i+1}$ or $a_{i+1}\rightarrow_1 a_{i}$, where $0\leq i \leq n-1$. We argue by induction on $n$.  If $n = 0$, then $a=b$ and there is nothing to prove. 
Suppose the statement  holds for strings of length $n-1$ and let $a=a_0, a_1,\dots, a_n=b$ be a string of length $n$. By induction, there is a $c$ such that $a_0\rightarrow c$ and $a_{n-1}\rightarrow c$. We consider two cases: 

If $b=a_n\rightarrow_1 a_{n-1}$, then clearly $a_0\rightarrow c$ and $b\rightarrow c$ and we are done. 

If $a_{n-1}\rightarrow_1 b$, then by the definition~(\ref{hfgtrgt655}), there is a vertex $v$ with weight $k$ such that $kv$ appears in the presentation of $a_{n-1}$, i.e., $a_{n-1}=kv+a'_{n-1}$ and $b=\rr(kv)+ a'_{n-1}$. 
Since $a_{n-1} \rightarrow c$ there is a string of transformations 
\begin{equation}\label{seqconfu}
a_{n-1}=kv+a_{n-1}' = c_0 \longrightarrow_1 c_1 \longrightarrow_1 \dots \longrightarrow_1 c_l=c.
\end{equation}
We consider two cases.  

\underline{Case 1.}   Going along the displayed string (\ref{seqconfu}) if there is no transformation of this $kv$ in each of the $c_i$, $0\leq i \leq l$,  then $kv$ appears in the presentation of each of the $c_i$ and therefore we have a chain 
\begin{equation}\label{firsttran}
a_{n-1}' = c_0' \longrightarrow_1 c_1' \longrightarrow_1 \dots \longrightarrow_1 c_l',
\end{equation}
where $c_i=kv+c_i'$, $0\leq i \leq l$. 

 Now we apply one more transformation on $c_l=kv+c_l' \rightarrow \rr(kv)+c_{l}'$. Since $b=\rr(kv)+a_{n-1}'$, applying the same transformations~(\ref{firsttran}) we have 
 \[b=\rr(kv)+a_{n-1}'\longrightarrow_1 \rr(kv)+c_1' \longrightarrow_1 \cdots \longrightarrow_1 \rr(kv)+c_{l-1}' \longrightarrow_1 \rr(kv)+c_l'\]
Therefore $a \rightarrow c=c_l=kv+c_l' \rightarrow \rr(kv)+c_l'$ and $b\rightarrow \rr(kv)+c_l'$ and we are done. 

\underline{Case 2.}  Suppose $0\leq s \leq l-1$ is the first instance in the string (\ref{seqconfu}) that $\rightarrow_1$ transforms this $kv$ to $\rr(kv)$. Thus we have a chain 
\begin{equation}\label{firsttran2}
a_{n-1}' = c_0' \longrightarrow_1 c_1' \longrightarrow_1 \dots \longrightarrow_1 c_s',
\end{equation}
where $c_i=kv+c_i'$ 
and \[a_{n-1}=kv+c_0'  \longrightarrow_1 kv + c_1' \longrightarrow_1 \dots \longrightarrow _1 kv+ c_s' \longrightarrow_1 \rr(kv)+c_s'=c_{s+1}\longrightarrow_1\dots \longrightarrow_1 c_l=c\]
Since $b=\rr(kv)+a_{n-1}'$, applying the same transformations~(\ref{firsttran2}) we have 
 \[b=\rr(kv)+a_{n-1}'=  \rr(kv)+c_0' \longrightarrow_1 \cdots \longrightarrow_1 \rr(kv)+c_s' =c_{s+1}\longrightarrow_1\dots \longrightarrow_1 c_l=c\]
Therefore $a\rightarrow c$ and $b\rightarrow c$ and we are done. 

The final statement follows from Remark \ref{sandpileex}.  
\end{proof}

Let $(E,w)$ be a weighted graph and  $H$ a hereditary 
subset of $E$.  Consider $(H,w)$, the weighted subgraph of $(E,w)$ consisting of all vertices of $H$ and all edges emitting from these vertices, with the same weights as in $E$. 
Observe that there is a well-defined monoid homomorphism $M(H,w)\rightarrow M(E,w); a \mapsto a$.  For $a,b \in M(H,w)$ if $a=b$ in $M(E,w)$, an application of the Confluence Lemma~\ref{aralem6} shows there is a $c\in \mathbb F_E$ such that $a\rightarrow c$ and $b\rightarrow c$. But since $H$ is hereditary,
all the transformations occur already in $H$, and thus $a=b$ in $M(H,w)$. Thus we can consider $M(H,w)$ as a submonoid of $M(E,w)$.   Consequently:  
 
\begin{remark}\label{Sisheredsat}
Let $(E,w)$ be a weighted graph.  Let $S$ denote the subset of $E^0$  consisting of those vertices which do not connect to any cycle in $E$.   Then clearly $S$ is hereditary. 
 So by the above observation we can consider the monoid $M(S,w)$ as a submonoid of $M(E,w)$. 
Furthermore, we note for later use that the graph $E/S$ is easily seen to contain no sinks.  
\end{remark}

As another application of the Confluence Lemma~\ref{aralem6}, we  determine the unit group of  reduced graph monoids.

\begin{proposition}\label{cansurv}
Let $(E,w)$ be a finite weighted graph, $M(E,w)$ its associated reduced graph monoid and $S$ the set of all vertices in $E$ which do  not connect to any cycle in $E$. 
Then $$Z(M(E,w))=M(S,w).$$   Furthermore,  if $(E,w)$ is a finite vertex weighted graph, then 
\begin{equation}\label{quowei}
M(E,w) \big / Z(M(E,w))  =  M(E,w) \big / M(S,w) \cong M(E/S,w_r),
\end{equation}
where $w_r$ is the restriction of the weight function $w$ to $E/S$.  
\end{proposition}
\begin{proof}
By Remark \ref{Sisheredsat}  we have $M(S,w)\subseteq M(E,w)$.  If $S$ is not empty, then any path emitting from an element $v$ in $S$ can be extended to a path that ends in a sink, as $E$ is finite and $v$ is not connected to a cycle.  An easy induction (on the maximum length of the paths connecting to sinks) shows that $kv=0$, for some $k\in \mathbb N^+$. Thus $M(S,w) \subseteq Z(M(E,w))$. 

Conversely, let $a\in Z(M(E,w))$. Then $a+b=0$ in $M(E,w)$  for some $b\in M(E,w)$. The Confluence Lemma~\ref{aralem6}  implies that $a+b\rightarrow 0$ in $\mathbb{F}_E$. Seeking a contradiction, suppose  that there is a vertex $w\in E^0\backslash S$ appearing in the presentation of $a$. Since $w$ is connected to a cycle, any possible transformations of $w$ or its multiple would give a vertex on a cycle and subsequently any further transformations always contain a vertex on a cycle. Thus $a+b$ can't be transformed to $0$, and so  $a\in M(S,w)$. 

The monoid isomorphism given in  (\ref{quowei}) (which is valid for any hereditary 
subset of $E$, of which $S$ is a particular case) is easy to check, and is left  to the reader. 
\end{proof}

We can now establish  the following corollary,  which will be used to connect sandpile monoids to weighted Leavitt path algebras in  Theorem~\ref{conicrit}. 

\begin{corollary}\label{corconic}
Let $M(E,w)$ be the reduced graph monoid associated to a finite weighted graph $(E,w)$ and $S$ the set of  vertices in $E$ which do not connect  to any cycle.

\begin{enumerate}[\upshape(i)]

\item $M(E,w)$ is a  group if and only if $E$ is acyclic. 

\item $M(E,w)$ is conical if and only if every non-sink vertex in $S$ has weight one. 
\end{enumerate}

\end{corollary}
\begin{proof}

(i) As noted above for arbitrary monoids,  $M(E,w)$ is a  group  if and only if $M(E,w)=Z(M(E,w))$. The claim follows now from Proposition~\ref{cansurv}. 

(ii) We have that $M(E,w)$ is conical if and only if $Z(M(E,w))=0$, so that by Proposition \ref{cansurv} we get $M(S,w)=0$. Suppose the weight  of each vertex in $S$ is $1$ (this in particular implies the weight of each edge emitting from the vertex is $1$).  Thus the relation given in Display~(\ref{monvertex}) for a non-sink vertex $v\in S$ reduces to $ v= \sum_{e \in s^{-1}(v)} r(e)$, whereas for a sink $s$, we have $s=0$. Thus for $v\in S$, an easy induction (on the maximum length of the paths connecting to sinks) shows that $v=0$ in $M(S,w)$. 
On the other hand if $M(E,w)$ is conical, then $M(S,w)=0$. Suppose $v\in S$. Since $v=0$ in $M(S,w)$, by the Confluence Lemma~\ref{aralem6}, $v\rightarrow 0$. However, no transformation can be applied to $v$ if the weight of $v$ is larger than $1$. 
This completes the proof. 
\end{proof}

\begin{remark}\label{conicalremark}  We note that if we consider a finite graph $E$ as a balanced weighted graph, then Corollary~\ref{corconic}(ii) says  that $M(E,w)$ is conical precisely when every  non-sink vertex in $S$ is irrelevant. 
\end{remark}

\begin{remark}\label{ConfluencedoesnotgiveRefinement}
 For those readers who are familiar with the Confluence Lemma of \cite[\S 6.3]{amp}, we note that  although the Confluence Lemma is a key step used in \cite{amp} to establish that $M_E$ is a refinement monoid (for unweighted graphs), and although we have established a corresponding   Confluence Lemma for weighted graphs (Lemma \ref{aralem6}), in general $M(E,w)$ need not be refinement.   
\end{remark}

\section{Unique normal form:  The Diamond Lemma}\label{diamondsec}

In addition to the Confluence Property, one of the main features of sandpile monoids is that their elements stabilise (i.e., are {\it reduction-finite}, see below) to a unique normal form.   In particular, sandpile monoids are necessarily finite.   
 Although any monoid of the form  $M(E,w)$ has been shown to satisfy the  Confluence Property in Lemma \ref{aralem6},  not every monoid of the form $M(E,w)$ has the property that each of its elements stabilises (nor do they need to be finite, as observed previously).   

\begin{example}\label{exinfiniteme}
 Let $E$ be the graph   

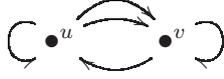
\begin{equation*}
\xymatrix{
 \bullet^u  \ar@(lu,ld)\ar@/^1.3pc/[r] \ar@/^0.7pc/[r] & \bullet^v \ar@/^0.9pc/[l] \ar@(ru,rd) 
}
\end{equation*}

\smallskip
\noindent on which we define the vertex weighting $w(u)=w(v)=2$.    
 Then the relations of reduced graph monoids
  give $2u=u+2v$ and $2v=u+v$ in $M(E,w)$.  One observes that the element $2u$ never  stabilises; specifically, we can continue to apply the two relations to the element $2u$ and produce any element of the monoid of the form $2u + nv$ for $n\in \mathbb{N}^+$.  
 
     We note further that the monoid $M(E,w)$ is  finite (equal to $\{0, u, v, 2u, 2v\}$), and  conical, and has atoms $u$ and $v$.   But $M(E,w)$ is not a refinement monoid, as an analysis of the equation $u+u = u+2v$ bears out.   
  
  \end{example}

  
\begin{example}\label{sizeofMdependsonweights}   Let $E$ be the graph pictured here.  

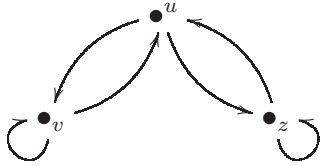
\[ \quad {
\def\labelstyle{\displaystyle}
\xymatrix{ {} & \bullet^{u}  \ar@/^{-10pt}/[rd] \ar@/^{-10pt}/ [ld] &  {} \\
\bullet_{v} \ar@(d,l) \ar@/^{-10pt}/[ru] &  & \bullet_{z} \ar@(d,r)
\ar@/^{-10pt}/ [lu] \\
}}
\]
\smallskip

\noindent
We point out that the size of $M(E,w)$ is highly dependent on the weight $w$ assigned to the edges of $E$.   If we assign weight $w_1(e) = 1$ to each edge (i.e., the unweighted case), then $w_1(u)=w_1(v)=w_1(z)=1$, and in this situation $M(E,w_1)$ is infinite.   (Note:  the monoid $M_E$ of this particular unweighted graph has been analyzed as Example $E_1^7$ in \cite[Section 4]{ClassQ}.)   Similarly, if we assign weight $w_2(e) = 2$ to each edge, then $w_2(u)=w_2(v)=w_2(z)=2$, and in this situation $M(E,w_2)$ is infinite as well.  (We will give the tools in Display \ref{Kzeroiso} below which will allow us to easily reach these  conclusions about $M(E,w_1)$ and $M(E,w_2)$.)   However, 
if we assign weight $w_3(e) = 3$ to each edge, then $w_3(u)=w_3(v)=w_3(z)=3$, and in this situation $M(E,w_3)$ is finite, indeed $\big| M(E,w_3) \big| = 27$.   This conclusion will follow from Theorem \ref{conicrit}, as we will see that $M(E,w_3)$ is isomorphic to the sandpile monoid $\SP(G)$ of an appropriate graph $G$.  
\end{example}

We will employ Bergman's Diamond Lemma to re-establish that the elements of  sandpile monoids have unique normal forms.   Although a proof of the uniqueness of normal forms has appeared in the literature (see e.g. \cite[Section 9]{BT}), we include a proof here for two reasons.  First, while previous proofs have made  oblique reference to  a ``Jordan-H\"{o}lder type" ``Diamond Lemma"  result stemming from  Newman \cite{N}, we will show that Bergman's Diamond Lemma can  be applied quite easily and directly here.  Second, with Example \ref{sizeofMdependsonweights} as context, we will play up exactly where in this analysis the specific weight function is utilised.   

   We briefly remind the reader of  the setting (see~\cite{bergman78}). 
Let $\mathbb F_X$ be the free commutative monoid generated by a nonempty set $X$. 
 Let $R$ be a set of pairs of the form $\sigma= (W_\sigma,f_\sigma)$, where $W_\sigma, f_\sigma\in \mathbb F_X\backslash \{0\}$.  The set $R$ is called a {\it reduction system} for $\mathbb F_X$.
For any $\sigma\in R$ and $A \in \mathbb F_X$, let $r_{A+\sigma}:\mathbb F_X \rightarrow \mathbb F_X$ denote a map that sends 
$A+W_\sigma$ to $A+f_\sigma$ and fixes all other elements of $\mathbb F_X$. The maps $r_{A+\sigma}: \mathbb F_X \rightarrow \mathbb F_X$ are called {\it reductions}. For $a,b\in \mathbb F_X$, we write $a\rightarrow b$ if there is a sequence  $r_1, r_2 ,\dots, r_i$ of reductions, such that $r_{i} \circ \cdots \circ r_1(a)=b$.
An element $a \in \mathbb{F}_X$ is called {\it reduction-finite} if for every infinite sequence $r_1, r_2 ,\dots$ of reductions, $r_i$ acts trivially on $r_{i-1} \circ \cdots \circ r_1(a)$, for all sufficiently large $i$. We  call an element $a\in \mathbb F_X$ {\it reduction-unique} if it is reduction-finite, and if its images under all final sequences of reductions are the same. This unique image is called the \emph{normal form} of the element $a$. 

For $\sigma \neq \tau$ in $R$,   we call a configuration in which    $W_\sigma= A+B$ and $W_\tau = B+C$ (for some  $A, B, C\in \mathbb F_X$)    an {\it overlap ambiguity} of  $R$.
 Such an overlap ambiguity  is called  {\it resolvable} if there exist compositions of reductions, $r$ and $r'$, such that $r(f_\sigma + C) = r'(A+ f_\tau)$. 
 Similarly, for  
 $\sigma\neq\tau$ in $R$, we call a configuration in which $W_\sigma = B$ and $W_\tau = A+B+C$   (for some  $A,B,C\in \mathbb F_X$)  an {\it inclusion ambiguity}.  Such an inclusion ambiguity is  called {\it resolvable}  if there exist compositions of reductions, $r$ and $r'$, such that $r(A+f_\sigma +C) = r'(f_\tau)$.  
 
  In order to use Bergman's Diamond Lemma result, it is necessary that $\mathbb F_X$ be  equipped with a {\it semigroup partial ordering} 
 $\leq$, i.e.,  if $x < x'$ then $x+y < x' + y$, for all $x, x', y \in \mathbb{F}_X$.  Furthermore, we need the partial ordering $\leq$ to be {\it compatible} with  $R$, namely, that  $f_\sigma <  W_\sigma$ for all $\sigma \in R$. 


{\bf Bergman's Diamond Lemma  in the setting of commutative monoids~\cite[Theorem~1.2, \S 9.1, \S10.3]{bergman78}}  says the following.     Suppose   $\leq$ is a semigroup partial ordering on $\mathbb F_X$ compatible
with the reduction system $R$, and that $\leq$ has  the descending chain condition.  Then  all ambiguities of $R$ are resolvable if and only if all elements of $\mathbb F_X$ are reduction-unique under $R$.

The main obstacle in trying to utilise  Bergman's machinery in practice  lies in establishing  that the ambiguities are resolvable. We are in  position to place sandpile monoids in the setting of Bergman's Diamond Lemma. We will see that neither type of  ambiguity arises in the context of sandpile relations, and thus once the partial order on the monoid is established, the fact that these monoids have reduction-unique forms will follow quite easily.   

\begin{proposition} \label{ReductionUnique}
Suppose $E$ is a sandpile graph.  Then every element of the sandpile monoid $\SP(E)$ is reduction-unique.
\end{proposition}
\begin{proof}
Let $E$ be a sandpile graph with the unique sink $s$ and consider $E$ as a balanced weighted graph by assigning $w(v)=|s^{-1}(v)|$ for all $v\in E^0 \setminus \{s\}$.  

Let $\mathbb{F}_E$ denote the free commutative monoid on the set $E^0$.   Let $R$ be the set of pairs of the form $\sigma= (W_v,f_v)$, 
 where 
 \begin{equation}\label{typreh}
 W_v=w(v) v,  \, \, \text{ and } \, \, f_v=\sum_{e\in s^{-1}(v)}r(e),
 \end{equation}
  for every non-sink vertex  $v$.

We employ here an idea quite similar to one presented in \cite{BT}.     Let $D=\max \{w(v) \ | \ v\in E^0\}$.  (Set $D=2$ in case this maximum value is $1$; i.e.,  in case each non-sink vertex in $E$ is irrelevant.) 
   Let $n$ denote $|E^0|$, and let $\ell_v$ denote the length of the shortest path which connects $v$ to the sink $s$.  Note that $\ell_v < n$ for all $v\in E^0$, as no shortest length path can contain a cycle.  
  The assignment  $v\mapsto D^{n-\ell_v}$, for any $v\in E^0$, induces a homomorphism of monoids $p:\mathbb F_E \rightarrow \mathbb N$, where   
  for $x = \sum_{v\in E^0} k_vv  \in \mathbb{F}_E$, we define
 \[p(x): = \sum_{v\in E^0} k_v D^{n-\ell_{v}}.\]

  
   We check that $p(W_\sigma) < p(f_\sigma)$ for the relations defined in  (\ref{typreh}).   By definition we have  $p(W_\sigma)=w(v)D^{n-\ell_v}$ and 
  $p(f_\sigma)= \sum_{e\in s^{-1}(v)}  D^{n-\ell_{r(e)}}$.   
  
  
   It is here that we use the assumption that $w(v) = |s^{-1}(v)| $.      Since $w(v)\leq D$ and  $\ell_{r(f)}=\ell(v)-1$  for some $f\in s^{-1}(v)$, we obtain 
 \begin{eqnarray*}
 p(W_\sigma) =  w(v)D^{n-\ell_v}  & =  & |s^{-1}(v)| D^{n-\ell_v}   \\
  & <  &  D \cdot  D^{n-\ell_{v}}+ \sum_{e\in s^{-1}(v)\backslash \{f \}} D^{n-\ell_{r(e)}}  \\
  & = &  \sum_{e\in s^{-1}(v)}  D^{n-\ell_{r(e)}}  \\ 
  & =  & p(f_\sigma).
  \end{eqnarray*}

\noindent
 We now define an order $\geq$ on $\mathbb{F}_E$ as follows:
 \[ x \geq y \ \ \  \mbox{ if } \  \ \  x=y \ \ \mbox{or} \ \ x  \rightarrow y,\]
where the binary relation $\rightarrow$ is defined as in Displays (\ref{hfgtrgt655}) and (\ref{hfgtrgt6551}). The  displayed inequality above 
yields that  $x>y$ implies $p(x) <p(y)$,  and so  $\geq$ is a partial order on $\mathbb{F}_E$.  By the definition of the ordering,  it is immediate that  $\geq$ is a semigroup partial order which is compatible with relations (\ref{typreh}).  
 The fortunate, obvious observation in this setting is that for $v \not= v' \in E^0$, the expressions  $W_v = w(v)v$ and $W_{v'} = w(v')v'$ in $\mathbb{F}_E$   involve neither overlap ambiguities nor inclusion ambiguities.  

Finally, we have $p(x) \leq  (\sum_{v\in E^0}  {k_v)}D^n$, for all $x =   \sum_{v\in E^0}  k_v v \in \mathbb{F}_E$.  This  guarantees that  the  partial ordering on $\mathbb F_E$ defined above has the  descending chain condition.

Thus Bergman's Diamond Lemma guarantees that the monoid generated by $\mathbb F_E$ subject to relations (\ref{typreh}) is reduction-unique. Since 
$\SP(E)$ is this monoid by identifying the sink to be zero, $\SP(E)$ is  reduction-unique, as desired.  
 \end{proof}



 
  
  
  

\begin{remark}\label{sizeofSP(E)} We note that the sandpile monoid $\SP(E)$ of a  sandpile graph $E$ must be finite.   Indeed $| \SP(E) | $ is the product of the positive integers $|s^{-1}(v_i)|$ taken over all non-sink vertices $v_i$ of $E$.   This is clear because each element of $\SP(E)$  corresponds to the (unique) reduced form of each of the elements in the quotient description of $\SP(E)$   as guaranteed by  Proposition \ref{ReductionUnique}. 

 \end{remark}

\section{Weighted Leavitt path algebras}\label{wLpasection}

In this section we briefly recall the notion of weighted Leavitt path algebras. These are algebras associated to weighted graphs (see \S\ref{weightedgraphssec}). We refer the reader to \cite{H} and \cite{Pre} for a detailed analysis of these algebras. 
\begin{deff}\label{weighteddef}
Let $(E,w)$ be a weighted graph and $\K$ a field.  The free $\K$-algebra generated by $\{v,e_i,e_i^*\mid v\in E^0, e\in E^1, 1\leq i\leq w(e)\}$ subject to relations
\begin{enumerate}[(i)]
\item $uv=\delta_{uv}u,  \text{ where } u,v\in E^0$,
\medskip
\item $s(e)e_i=e_i=e_ir(e),~r(e)e_i^*=e_i^*=e_i^*s(e),  \text{ where } e\in E^1, 1\leq i\leq w(e)$,
\medskip
\item 
$\sum_{e\in s^{-1}(v)}e_ie_j^*= \delta_{ij}v, \text{ where } v\in E_{\reg}^0 \text{ and } 1\leq i, j\leq w(v)$, 
\medskip 
\item $\sum_{1\leq i\leq w(v)}e_i^*f_i= \delta_{ef}r(e), \text{ where } v\in E_{\reg}^0 \text{ and } e,f\in s^{-1}(v)$,
\end{enumerate}
is called the {\it weighted Leavitt path algebra} of $(E,w)$, and denoted $L_\K(E,w)$. In relations (iii) and (iv) we set $e_i$ and $e_i^*$ to be zero whenever $i > w(e)$.    
\end{deff}

Note that if the weight of each edge is $1$, then $L_\K(E,w)$ reduces to the usual Leavitt path algebra $L_\K(E)$  of the graph $E$.   Also, if $E$ is the empty graph then $L_\K(E,w)$ is defined to be the zero ring.  

\begin{example}\label{exweighted}




 
Weighted Leavitt path algebras were originally  conceived in \cite{H}  in order to provide a context in which to generalise the algebras $L_\K(n,n+k)$ constructed by W. Leavitt
mentioned in the Introduction.  
Specifically, let $E_{n,n+k}$ denote the  weighted graph consisting of one vertex $v$ and $n$ loops of weight $n+k$ at $v$.   Then it is shown in \cite[Example 5.5]{H} that 
$$L_\K(E_{n,n+k}) \cong L_\K(n,n+k).$$

\noindent
In particular, let $(E,w)$ be the weighted graph consisting of one vertex and one loop of weight $n$. 
Then the weighted Leavitt path algebra of $(E,w)$ is isomorphic to the Leavitt path algebra of a graph with one vertex and $n$ loops, which in turn is the Leavitt algebra $L_\K(1,n)$. 
\end{example}

\begin{remark}  \label{remarkongrading}
(1)  Let $n,k\in \mathbb{N}^+$, and let $(E,w)$ denote the weighted graph $E_{n,n+k}$ described in Example \ref{exweighted}.  So $(E,w)$ has one vertex $v$, and $n$ loops at $v$ each having weight $n+k$.   Let $(E',w')$ denote the weighted graph  $E_{n+k,n}$.  So $(E',w')$ has  one vertex $v'$, and $n+k$ loops at $v'$ each having weight $n$.   Because of the symmetry involved in the definition of weighted Leavitt path algebras, it is easy to show that $L_\K(E,w) \cong L_\K(E',w')$ as $\K$-algebras.   Our choice here to focus on the $n$ loops of weight $n+k$ point of view is to make the connection we will establish in Theorem \ref{conicrit} below between sandpile monoids and weighted Leavitt path algebras more transparent.   We note, however, that the algebras  $L_\K(E,w)$ and $L_\K(E',w')$ are different if these are viewed as  $\mathbb{Z}$-graded  $\K$-algebras in the standard Leavitt path algebra  $\mathbb{Z}$-grading

(2) The previous paragraph notwithstanding, we caution the reader that in general one cannot cavalierly modify the edges and weights and expect that the corresponding Leavitt path algebras will be isomorphic.   For instance, let $(F,w)$ be the graph with one vertex $v$, and two loops at $v$ each having weight $2$;  and let $(F',w')$ be the graph with one vertex $v'$ and four loops at $v$ each having weight $1$.   Then $L_\K(F,w) \not\cong L_\K(F',w')$.  (Indeed, $M(F,w) \not\cong M(F',w')$ as monoids, since  the latter is finite,
while the former is infinite.  Since these two monoids are not isomorphic, neither can the two weighted Leavitt path algebras be isomorphic; see 
Theorem \ref{kspcts} below.)    
\end{remark}  


For a ring $A$ with identity, the commutative monoid $\V(A)$ is defined as the set of isomorphism classes of finitely generated projective right $A$-modules equipped with  direct sum as the binary operation. The group completion of this monoid is the Grothendieck group $K_0(A)$ of $A$. We refer the reader to the book of Magurn for a comprehensive treatment of 
these ideas 
~\cite{magurn}. 

The Leavitt algebras $L_\K(n,n+k)$ can be produced from Bergman's machinery of the Universal Construction of Rings~\cite{B}, which subsequently thereby also describes the structure of the $\mathcal{V}$-monoids  of these algebras. Let $\K$ be a field, $A$ a $\K$-algebra and let $P$ and $Q$ be nonzero finitely generated projective right $A$-modules. Then there is a $\K$-algebra $B$, with an algebra homomorphism $A\rightarrow B$ such that there is a universal isomorphism $i:\overline P \rightarrow \overline Q$, where $\overline M=M\otimes_A B$ for any right $A$-module $M$ (\cite[p. 38 and Theorem~3.3]{B}).  Bergman's Theorem  \cite[Theorem 5.2]{B} states that $\V(B)$ is the quotient of $\V(A)$ modulo the  relation $[P]=[Q]$.  

Using this, starting with a field $A=\K$, and the finitely generated projective $A$-modules $P=\K^n$ and $Q=\K^{n+k}$, Bergman's machinery applied to this data gives that  $B=L_\K(n,n+k)$, and subsequently that 
$$\mathcal{V}(L_\K(n,n+k))\cong M_{n,n+k} =  \{0, x, 2x, \dots, (n+k-1)x \ | \ (n+k) x = nx \}.$$  
 We note that the process of explicitly identifying the $\mathcal{V}$-monoid $\mathcal{V}(B)$  is extremely delicate, requiring the full use of the powerful tools of Bergman developed in \cite{B}.  Indeed, while Leavitt had established in \cite{vitt62} that the monoid of {\it free} modules over $L_\K(n,n+k)$ is isomorphic to $M_{n,n+k}$, it was not established until Bergman's fundamental work fifteen years later in \cite{B}  that  $M_{n,n+k}$  actually represents the full $\mathcal{V}$-monoid of $L_\K(n,n+k)$.

\begin{remark}
It is germane to point out that the original observation which led us to connect  sandpile monoids to the $\mathcal{V}$-monoids of weighted Leavitt path algebras was made in the context of the monoid $M_{n,n+k}$.  Specifically, $M_{n,n+k}$ arises on the one hand  (as noted in Example \ref{exweighted} together with the previous paragraph) as the $\mathcal{V}$-monoid of the most elementary type of vertex weighted graph, namely  the vertex weighted graph having exactly one vertex.  On the other hand (as noted in Example \ref{Mnn+kissandpile}), $M_{n,n+k}$ also arises as the sandpile monoid of the most elementary type of sandpile graph, namely  the sandpile graph having exactly one non-sink vertex.    Once this connection was realised, the roadmap to connect general sandpile monoids and weighted Leavitt path algebras became evident.   
\end{remark}

Although the  computation of the $\mathcal{V}$-monoid  can be carried out for the Leavitt path algebra of arbitrary (finite) weighted graphs (see e.g. \cite{P}), for the current work we need only describe the $\mathcal{V}$-monoid of weighted Leavitt path algebras associated to vertex weighted graphs, which we now do.  

\begin{deff}\label{MsubEwdef}
Let $(E,w)$ be a row-finite vertex weighted graph. 
We define the {\it graph monoid} of $(E,w)$ to be 
\begin{equation}\label{monvertexnosink}
M_{(E,w)} \  \ := \ \ \mathbb F_E \  \Big / \ \Big \langle w(v) v= \sum_{e \in s^{-1}(v)} r(e) \, \big | \, v\in E_{\reg}^0\Big \rangle.
\end{equation}
\end{deff}
\begin{remark}\label{M(E,w)equalMsub(E,w)}
While clearly similar one to the other, the definitions of the monoids $M(E,w)$ (Definition~\ref{genmon}) and $M_{(E,w)}$ are not identical.  Specifically, there is no relation in $M_{(E,w)}$ whose left hand side involves any sinks which may exist in $E$; while in $M(E,w)$, there is a relation of the form $1\cdot s=0$ for each sink in $E$.    Formally,
 \[M(E,w)  \ \cong \ M_{(E,w)} \big /   \langle s=0  \ \big | \, s\in E_{\ssink}^0 \rangle. \]
 So if $E$ contains no sinks, then $M(E,w) = M_{(E,w)}$.  
\end{remark}

In particular, for a directed (unweighted) row-finite  graph $E$  we have the  \emph{graph monoid}
\begin{equation}\label{weightonemon}
 M_E\cong \mathbb F_E \Big / \Big \langle  v= \sum_{e \in s^{-1}(v)} r(e) \, \big | \, v\in E_{\reg}^0\Big \rangle.
\end{equation}
originally defined in \cite{amp}.    See also \cite[Definition 1.4.2]{TheBook}.  

Recall that for a commutative monoid $M$, an element $\mathbf{1} \in M$ is called a \emph{distinguished} element if for any element $m\in M$, there exist $k \in \mathbb N$ and $m'\in M$ such that $m+m'=k\mathbf{1}$.

In particular, let $(E,w)$ be  a finite vertex weighted graph, and let $M_{(E,w)}$ the monoid  described in Definition \ref{MsubEwdef}.    Then for each choice of integers $0< k_v <w(v)$ (for each non-sink vertex $v$ of $E$), the element  $\sum_{v\in E^0}  k_v v$ of $M_{(E,w)}$ is easily seen to be  a distinguished element of $M_{(E,w)}$. 

In \cite{amp}, Ara, Moreno, and Pardo built Leavitt path algebras $L_\K(E)$  via the  machinery presented by Bergman in \cite{B}, 
They showed that $L_\K(E)$ can be obtained by Bergman's universal construction of  rings, 
which thereby allowed them to conclude in particular that $\mathcal V(L_\K(E))\cong M_E$ (\cite[Theorem 3.5]{amp}).  Similar arguments were used  in~\cite{H,P,Pre} to construct weighted Leavitt path algebras, and thereby realise their corresponding monoids, via this same machinery.  Here    we describe  in detail   how to  produce weighted Leavitt path algebras  via the Bergman machinery, in the situation where one starts with a monoid of the form $M_{(E,w)}$.    Additionally, we will  show how  to modify the resulting algebra in such a way that the regular module of the algebra corresponds to whatever distinguished element of $M_{(E,w)}$ one might choose.  


\begin{theorem}\label{kspcts} 
Let $\K$ be a field, $(E,w)$ a finite vertex weighted graph, and $M_{(E,w)}$ the monoid  described in Definition \ref{MsubEwdef}.   For each choice of integers $0< k_v <w(v)$ (for each non-sink vertex $v$ of $E$),
 let $\mathbf{1}$ denote the distinguished element $\sum_{v\in E^0}  k_v v$ of $M_{(E,w)}$.  Consider the vertex weighted graph $(F,w)$ obtained from $E$ by adjoining a path of length $k_v-1$ to each vertex $v$ in $E$, and assign weight $1$ to  each newly adjoined vertex.   
   Then there is a natural monoid isomorphism 
$$\V(L_\K(F,w))\cong M_{(E,w)},$$
 via which
 $  [L_\K(F,w)]  \ \mbox{corresponds  to} \  \mathbf{1}.$
 In particular,  for any finite vertex weighted graph $(E,w)$,   there is a natural isomorphism 
$$\mathcal V(L_\K(E,w)) \cong  M_{(E,w)},$$
via which 
$ [L_\K(E,w)]  \ \mbox{corresponds  to} \  \sum_{v\in E^0}v. $
\end{theorem}

\begin{proof}

Since $F$ is an enlarged version of the graph $E$, we can name the corresponding vertices in $F$ by $v^0$ in case $v\in E^0$, and name the vertices on the path of length $k_v-1$ connected to the vertex $v^0$ by $v^i$, $1\leq i \leq k_v-1$.     (For any vertex $v$ having $k_v = 1$ we do not adjoin any new edges to $v$, but we do relabel $v$ as $v^0$.)   






 It is easy to show that 
$$
\phi: M_{(F,w)} \longrightarrow M_{(E,w)} \ \ \mbox{via} \ \ 
v^i \longmapsto v
$$
induces an isomorphism of monoids with $\phi(\sum_{v\in F^0} v)=\mathbf{1}$. 

Next we use Bergman's universal construction machinery of~\cite{B} to show that $M_{(F,w)}$ can be realised as the monoid $\V(L_\K(F,w))$ with $\sum_{v\in F^0} v\in M_{(F,w)}$ corresponding to $[L_\K(F,w)]\in \V(L_\K(F,w))$.  Together with the just-noted monoid isomorphism $\phi$, this will give the desired result.  


Let $A_0=\prod_{v\in F^0}\K$. Thus the monoid $\mathcal V(A_0) \cong \bigoplus_{v\in F^0} \mathbb N$  is the free commutative monoid generated by vertices of the graph $F$, with $[A_0]=\sum_{v\in F^0} v \mapsto  (1,1,\dots,1)$. We will use Bergman's Theorems 3.1 and 3.2 of \cite{B} to construct a universal ring $A$ from $A_0$ which identifies certain projective modules of $A_0$. Then Theorem 5.2 of \cite{B} guarantees $\mathcal V(A)$ is isomorphic to $\mathcal V(A_0)$ modulo the identifications of these projective modules. The projective modules over $A_0$ will be  chosen to give us precisely the relations of Definition \ref{MsubEwdef}, so that  $\mathcal V(A)\cong M_{(F,w)}$. On the other hand we will observe that the ring $A$ is precisely $L_\K(F,w)$. 

Let $\{v_1,\dots,v_m\}$ be the set of all the vertices in $F$ which emit edges.  Consider the following finitely generated projective $A_0$-modules: \  $P=\textstyle{\bigoplus_{w(v_1)}}(A_0v_1)$,  and $Q=\textstyle{\bigoplus_{\{e \in F^1 \mid s(e)=v_1 \}}}A_0r(e)$.  (We are identifying  $v_i$ with the $i$-th unit vector in $A_0=\prod_{v\in F^0}\K$.)   Using  Bergman's  \cite[Theorem 3.1]{B}, one obtains a ring $A_0'$ and universal homomorphisms $i:A_0' \otimes _{A_0} P\rightarrow A_0' \otimes _{A_0} Q$ and $\overline i: A_0' \otimes _{A_0} Q\rightarrow A_0' \otimes _{A_0} P$. Next, Bergman's \cite[Theorem 3.2]{B} applies to obtain the ring $A_1$, where extensions of $i$ and $\overline i$ (call them again by $i$ and $\overline i$) over $A_1$ gives $1-i\overline i=0$ and $1-\overline i i=0$. Thus we have  $A_1=A_0 \langle i,i^{-1}: \overline P \cong \overline Q\rangle$ with a universal isomorphism \[i: \  \overline P:=A_1\otimes_{A_0} P \ \  \rightarrow  \ \ \overline  Q:=A_1\otimes_{A_0} Q.\] 

Continuing to follow the Bergman construction (specifically, the proofs of Theorem 3.1 and 3.2 in \cite{B})  shows that $A_1$ is $L_\K(X_1,w)$, where $X_1$ is a graph with the same vertices as $F$ and where $v_1$ emits the same edges as in $F$ and other vertices do not emit any edges. Namely, if $\{ e_1,\dots,e_s\}$ is all the edges which are emitted from $v_1$ with $n=w(v_1)$ then the right multiplication by  the matrix  $Y=(e_{ij})_{1\leq j \leq s, 1\leq i\leq n}$, where $e_{ij}={(e_j)}_i$, gives the map 
\[i: \ \overline P=w(v_1)(A_1 v_1) \ \  \rightarrow \ \  \overline Q =\textstyle{\bigoplus_{\{e \in F^1 \mid s(e)=v_1 \}}} A_1r(e),\] while $X=(Y^*)^t$, where ${}^t$ is the transpose operation,  gives $i^{-1}$. Now \cite[Theorem~5.2]{B} guarantees that $\V(A_1)$ is obtained from $\V(A_0)$ by adding the relation $[P]=[Q]$.  Translating this to our setting, we get that $\V(A_1)$ is the monoid generated by the set $\{[v] \mid v\in F^0\}$ subject to the relation $w(v_1) [v_1]=\sum_{\{e \in F^1 \mid s(e )=v_1 \}} [r(e)].$ 

We repeat this process to cover the whole graph. To be precise, let $A_k=L_\K(X_k,w)$, $k\geq 1$, where $X_k$ is the graph with the same vertices as $F$, but only the first $k$ vertices $\{v_1,\dots,v_k\}$ emit structured edges. By induction, $\V(A_k)$ is a commutative monoid generated by $\{[v] \mid v\in F^0\}$ subject to the relation 
$w(v_i) [v_i]=\sum_{\{e\in F^1 \mid s(e)=v_i \}} [r(\alpha)]$, where $1\leq i \leq k$. Then  $A_{k+1}=A_k \langle i,i^{-1}: \overline P \cong \overline Q\rangle$ with $P=\textstyle{\bigoplus_{w(v_{k+1})}}A_{k} v_{k+1}$ and $Q =\textstyle{\bigoplus_{\{e \in F^1 \mid s(e)=v_{k+1}\}}}A_k r(e)$.  So one more application of \cite[Theorem~5.2]{B} gives that  $\V(A_{k+1})$ is the monoid generated by all the vertices of $F$ subject to relations corresponding to $\{v_1,\dots,v_{k+1}\}$. Thus after repeating this process $m$ times we arrive at the monoid $\mathcal V(A_0)$ subject to exact same relations of Definition \ref{MsubEwdef}. 
Putting these together we have 
$\mathcal V(L_\K(F,w)) \cong M_{(F,w)}\cong M_{(E,w)}$, as desired.  

The final statement is simply the case where $k_v$ is chosen to be $1$ for all non-sink vertices $v$ of $E^0.$  We note that a proof of the final statement appears as 
\cite[Theorem 5.21]{H} (without mentioning the distinguished element). 
\end{proof}




For any finite graph $E$ we let $A_E$ denote the adjacency matrix of $E$,  i.e., $(a_{ij}) \in 
M_{|E^0|}({\mathbb N})$ where $a_{ij}$ is the
number of edges from $v_i$ to $v_j$.

\begin{remark}\label{computeGrothgroup}
 For a finite graph $E$, let
$I'_w$ be the \emph{weighted identity matrix of $(E,w)$}, defined to be the $|E^0| \times |E^0|$  matrix $(b_{ij})$, where
$b_{ij}=0$ for $i\not = j$ and $b_{ii}=w(v_i)$.
Denote by $N^t$ and $I_w$, respectively,  
the matrices obtained from $A_E$ and $I'_w$ by first taking the transpose, and then removing the columns
corresponding to sinks, respectively. Since the Grothendieck group of $L_\K(E,w)$ is the group completion of the commutative monoid $M_{(E,w)}$,  we obtain that 
\begin{equation}\label{Kzeroiso}
K_0(L_\K(E,w))\cong \coker\big(N^t-I_w:\mathbb Z^{E_{\reg}^0} \longrightarrow \mathbb Z^{E^0}\big).
\end{equation}

\noindent

\end{remark}

\bigskip

Let $(E,w)$ be any vertex weighted graph.   If $E$ has no sinks, then $M_{(E,w)} = M(E,w)$.   On the other hand, if $E$ does contain sinks, then one may construct the weighted graph $E'$ by attaching a loop to each sink, and defining  the weight $w'$ on $E'$ to be $w$ for any edge $e\in E^1$, and $w'(e) = 1$ for any added loop.   Then clearly $M_{(E,w)} = M(E', w')$.  Consequently, by the Confluence Lemma \ref{aralem6} together with Theorem \ref{kspcts},  we get 

\begin{corollary}\label{Vmonoidshaveconfluence}
Let $(E,w)$ be a weighted graph.   Then  $\mathcal{V}(L_\K(E,w))$ has the Confluence Property.  
\end{corollary}


\section{Sandpile monoids and weighted Leavitt path algebras: the connection}\label{connectionsection}

Let $E$ be a row-finite graph,  and let $S$  denote the (hereditary)  subset of vertices of $E$ which do not connect to any cycle. 
Consider $E$ as a balanced weighted graph $(E,w)$, so that $w(v)=|s^{-1}(v)|$ for all $v\in E_{\reg}^0$,  and consider $(E/S,w_r)$ as a weighted subgraph of $(E,w)$ as described in Remark \ref{weightonquotientremark}.  

We are now in  position to realise the conical sandpile monoids as the $\mathcal{V}$-monoids of weighted Leavitt path algebras.

\begin{theorem} \label{conicrit}
Let $E$ be a sandpile graph,  and let $S$ denote the set of vertices of $E$ which do not connect to any cycle.  Let $w$ denote a balanced weighting on $E$.  
 Then 
 \begin{equation}\label{equhi}
 \SP(E) \big / Z(\SP(E)) \cong \mathcal V(L_\K(E/S,w_r)). 
\end{equation}
In particular  $\SP(E)$ is conical if and only if  any vertex not connected to a cycle is irrelevant. In this case
\begin{equation}\label{mainisomorphism}
 \SP(E)  \cong \mathcal V(L_\K(E/S,w_r)) \ \ \ \ \ \mbox{and} \ \ \ \ \  \mathcal{G}(E)\cong K_0(L_\K(E/S,w_r)).
\end{equation}
\end{theorem}
\begin{proof}
As noted in Remark  ~\ref{sandpileex},   $\SP(E) \cong M(E,w)$.  
By Proposition~\ref{cansurv}, 
$\SP(E) \big / Z(\SP(E)) \cong M(E/S,w_r).$  By the definition of $S$ we see that $E/S$ has no sinks.  So $M(E/S, w_r) = M_{(E/S, w_r)}$.   
But by  Theorem \ref{kspcts} applied to the graph $E/S$, we get  $M_{(E/S,w_r)}\cong \mathcal V(L_\K(E/S,w_r))$. This gives the desired isomorphism (\ref{equhi}). 

By definition, the Grothendieck group $K_0(A)$ of any associative unital ring $A$ is  the group completion of  $\mathcal{V}(A)$.   So the second part of the statement follows from 
Remark  ~\ref{sandpileex} and Lemma~\ref{grothenlem} (as $\mathcal{G}(E)$ is by definition the smallest ideal of $\SP(E)$). 
\end{proof}

The statement that the conicality of a sandpile monoid is equivalent to the property that  vertices not connected to cycles are irrelevant has been established previously in a number of articles, see e.g. ~\cite[Proposition 4.2.5]{T} and \cite[Proposition 5.7]{BT}.

Here are a few examples which clarify how Theorem \ref{conicrit} plays out in some specific important situations. 
\begin{example}\label{VmonoidisoMnn+kexample}
Let $G$ denote the graph of Example \ref{Mnn+kissandpile}.   So $S = \{s\}$.    Then $(G/S, w_r)$ is the graph $E_{n,n+k}$ (defined in Example \ref{weightedgraphmonoidexample}) 
having one vertex $v$, and $n$ loops at $v$ each of weight $n+k$.   So by Example \ref{exweighted}   we have that $L_\K(G/S, w_r) \cong L_\K(n,n+k)$.    Then  Theorem \ref{conicrit} (re)establishes 
 that $$\mathcal{V}(L_\K(n,n+k)) \cong \mathcal{V}(L_\K(G/\{s\}, w_r)) \cong \SP(G) \cong M_{n,n+k} \ ,$$
 as expected.  
\end{example}

\begin{example}\label{exinfiniteme2}
Let $(E,w_3)$ be the weighted graph described in Example \ref{sizeofMdependsonweights}.  
 Now consider the sandpile graph $T$ pictured here.   
 \[ \quad {
\def\labelstyle{\displaystyle}
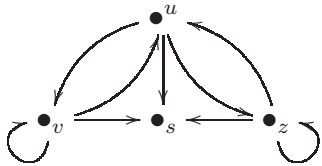
\xymatrix{ {} & \bullet^{u}  \ar@/^{-10pt}/[rd] \ar@/^{-10pt}/ [ld] \ar[d] &  {}\\
\bullet_{v} \ar@(d,l) \ar@/^{-10pt}/[ru] \ar[r]&  \bullet_{s} & \bullet_{z} \ar@(d,r) \ar[l]
\ar@/^{-10pt}/ [lu] \\
}}
\]

\smallskip
\noindent
Then clearly as graphs  $E = T/S$.   But $w_3$ on $E= T/S$ is precisely $w_r$ inherited from $T$, where $w$ is the balanced weight function.   So $M(E,w_3) = \SP(T)$.      In particular,  by Theorem \ref{conicrit},   $\SP(T)$ arises as the $\mathcal{V}$-monoid of the Leavitt path algebra of the weighted graph $(E, w_3).$ 

 Remark~\ref{sizeofSP(E)} gives that the size of this monoid is $3\times 3 \times 3=27$.

\end{example}

\begin{example}\label{prime}
Let $E$ be a reduced sandpile graph with $|SP(E)|=p$, a prime number. We show that either $\SP(E)\cong \mathbb Z_p$,  
or $\SP(E)\cong M_{p-l,p} \cong  \mathcal V(L_\K(E_{p-l,p}))$, for some $1\leq l<p$. 

Since the number of elements of the sandpile monoid is prime, by Remark \ref{sizeofSP(E)}   the graph $E$ must have only two vertices, namely a sink $s$ and a non-sink vertex $v$ which connects to the sink.

 If all the edges from $v$ are connected to the sink, then by the relations that define a sandpile monoid, $\SP(E)\cong M(E,w) \cong \langle v \mid pv=0\rangle$, where $(E,w)$ is the balanced weighted graph. Clearly then $\SP(E)\cong \mathbb Z_p$.   We note that in this case $\SP(E)$ is not conical (as, e.g., $1 + (p-1) = 0$), so cannot be realised as the $\mathcal{V}$-monoid of any associative ring.

 On the other hand, if $E$ consists of $p-l$ loops at $v$ and $l$ edges from $v$ to  $s$,  the relations of the monoid of the balanced weighted graph gives $\SP(E)\cong M(E \setminus \{s\}, w_r)\cong \langle v \mid pv=(p-l)v\rangle = M_{p-l,p}$. But this is exactly $\mathcal V(L_\K(E_{p-l,p}))$  by Example \ref{VmonoidisoMnn+kexample}. 

Consequently, if $|SP(E)|$ is prime, then either $\SP(E)$ is not conical, or, in case $\SP(E)$ is conical, then by the observation made in Example  \ref{Mplusexample},    $\mathcal{G}(E)$ is a cyclic group of order $p-l$ for some $1\leq l < p$.   

Now let $F$ be the graph $$ \quad {
\def\labelstyle{\displaystyle}
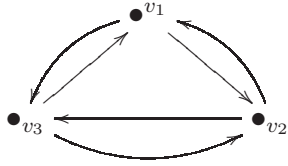
\xymatrix{ {} & \bullet^{v_1}  \ar[rd] \ar@/^{-10pt}/ [ld] &  {} \\
\bullet_{v_3}  \ar[ru] \ar@/^{-15pt}/ [rr]&  & \bullet_{v_2}
 \ar[ll]
\ar@/^{-10pt}/ [lu] \\
}}
$$

\medskip
\noindent
The  graph $F$ appears in \cite[Example 3.8]{ClassQ}, where it is shown that $|\mathcal{V}(L_\K(F))| = 5$, and that $K_0(L_\K(F))$ is isomorphic to the non-cyclic group $  \mathbb{Z}_2 \times \mathbb{Z}_2$.    In particular, $F$ is a graph for which $\mathcal{V}(L_\K(F))$ is finite, but for which $\mathcal{V}(L_\K(F))$ is not isomorphic to $\SP(E)$ for any  sandpile graph $E$.   
\end{example}

\begin{example}\label{weightedcycle}
Let $E$ be a \emph{weighted cycle}.  That is, $E$ is  a vertex weighted graph $(E,w)$ for which $E^0 =  \{v_1, v_2, \dots , v_m\}$, $E^1 = \{e_1, e_2, \dots , e_m\}$ with $s(e_i) = v_i$ and $r(e_i) = v_{i+1}$ for $1\leq i \leq m-1$,  $s(e_m)=v_m$, $r(v_m)=e_1$, and $w(e_i) = w(s(e_i))$ is some positive integer for each $1\leq i \leq m$.   We assume (for reasons described below) that $w(v_j) \geq 2$ for at least one $v_j$.  

  Consider the unweighted directed graph $F:=\hat E^{\op}$.   Here $\hat E$ denotes the (unweighted)  directed graph obtained from $E$ by replacing any edge in $E$ of weight $w(e) >1$  with  $w(e)$ edges;  and  $\op$ stands for the opposite of the graph, i.e.  the  graph with all edge orientations reversed.   Comparing the generators and relations of the algebra $L_\K(E,w)$ with those of $L_\K(F)$ shows that  $L_\K(E,w) \cong L_\K(F)$ as $\K$-algebras.

  Therefore by Theorem \ref{kspcts} we get 
  \[M_{(E,w)}\cong \mathcal V(L_\K(E,w)) \cong \mathcal V(L_\K(F)) \cong M_F.\]
  
  Now consider the graph $G$ defined as follows.  The vertex set of $G$ is $ G^0 := \{ s, v_1, v_2, \dots , v_m \}$.   The vertex $s$ is a sink in $G$.   The edge set $G^1$ consists of the following.   As in $E$, for each $v_i \in G^0$, $s^{-1}(v_i)$ includes the edge $e_i$.  In addition, in $G$, each $s^{-1}(v_i)$ also includes   $w(v_i)-1$ edges in $G^1$ from $v_i$ to $s$.    The assumption that $w(v_j)\geq 2$ for some $j$ ensures that there is at least one path from each $v_i$ to $s$; so $G$ is indeed  a sandpile graph.   
  
    Then by considering the generators and relations of the indicated monoids, we get 
\begin{equation}\label{threeisos}
\SP(G) \cong V(L_\K(E,w)) \cong \mathcal V(L_\K(F)).
\end{equation}

\noindent
In particular, $\SP(G)$ arises as the $\mathcal{V}$-monoid of an unweighted Leavitt path algebra.

For any positive integer $n$ let $C_n$ denote the monoid 
$$C_n = \{ 0, x, 2x, \dots, (n-1)x \}  \ \mbox{ with relation } nx = x.$$
So $C_n = M_{1,n}$ when $n\geq 2$.   Then   each of the three displayed monoids in (\ref{threeisos})  is isomorphic to the monoid $C_N$, where $N = w(v_1) \cdot w(v_2) \cdots w(v_m)$.     We note that if $n=1$ then the monoid $C_1$ is infinite;  thus the condition $w(v_i) \geq 2$ for some $i$ is necessary to ensure that this monoid is finite.  
\end{example}

As a specific instance of the ideas presented in Example \ref{weightedcycle}, suppose $(E,w)$ is the weighted cycle indicated below, where $E^0 = \{a,b,c\}$,  $w(a) = w(b) = 2$ and $w(c)=1$.    The associated (unweighted) graphs $F = \hat{E}^{op}$ and $G$ are pictured as well.        Since $N = 2\cdot 2\cdot 1 = 4$ we get that each of the monoids in (\ref{threeisos}) is isomorphic to $C_4$.    Of course this is not hard to see directly:  note that in $M(E,w)$ we have $2a = b, 2b = c,$ and $c=a$, and thus $M(E,w) = \{0,a,2a,3a\} $, where $4a = a$.

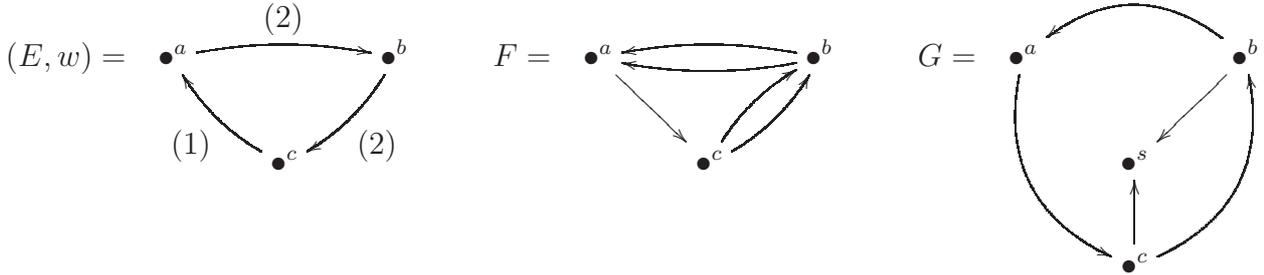
 $$  (E,w) =  \ \ {
\def\labelstyle{\displaystyle}
\xymatrix{ \bullet^{a} \ar@/^{5pt}/ [rr]^{(2)}  
   & &   \bullet^{b} \ar@/^{5pt}/ [dl]^{(2)} 
 \\   & \bullet^{c} \ar@/^{5pt}/ [lu]^{(1)}
 & 
}} \hskip1cm  F =  \ \ {
\def\labelstyle{\displaystyle}
\xymatrix{ \bullet^{a} \ar[rd]  
   & &   \bullet^{b} \ar@/^{5pt}/ [ll] 
   \ar@/_{5pt}/[ll]
 \\   & \bullet^{c} \ar@/^{5pt}/ [ru]
  \ar@/_{5pt}/[ru] 
 & 
 }}
\hskip1cm G =  \ \ {
\def\labelstyle{\displaystyle}
\xymatrix{ \bullet^{a}     \ar@/_{20pt}/ [ddr]
   & &   \bullet^{b} \ar@/_{20pt}/ [ll] \ar[dl]
 \\   
 & \bullet^{s} 
  \\
  & \bullet^{c} \ar@/_{20pt}/ [uur] \ar[u] & 
}} \hskip1cm 
$$


In Theorem \ref{conicrit}    we  established the ubiquity of weighted Leavitt path algebras in the context of sandpile monoids, namely, that every conical sandpile monoid $\SP(G)$ arises as the $\mathcal{V}$-monoid of a suitably defined  weighted Leavitt path algebra.   Perhaps surprisingly, the {\it un}weighted Leavitt path algebras (i.e., the classical Leavitt path algebras) turn out to be quite sparse in this context.  We will establish the precise statement in Proposition \ref{V(L(E))issandpile}.  


 We now assume that  $G$ denotes a conical sandpile graph; i.e., $G$ contains a unique sink $s$, and every vertex in $G$ connects to $s$, and (by Theorem \ref{conicrit}) every relevant vertex connects to a cycle.   


For context, we note that it is easy to find examples of reduced conical sandpile graphs $G$ for which there exists $v\in G^0 \setminus \{s\}$ and  $x\in \mathbb{F}_{G}$  having $x\neq v$ in $ \mathbb{F}_{G}$, but $x=v$ in $\SP(G)$.     For instance, let $v,w \in G^0$ for which $s^{-1}(w) = \{e, f_1, \dots , f_k\}$ where  $k\geq 1$ and $r(e) = v$ and  $r(f_i) = s$ for all $1\leq i \leq k$.    Then $ x := (k+1)w \neq v$ in $ \mathbb{F}_{G}$ (even if $v=w$), because $v$ is an atom in $ \mathbb{F}_{G}$.  But $x = v$ in $\SP(G)$.    

We will now show that the example described in the previous paragraph essentially describes the only possible such configurations.   

\begin{lemma}\label{vertexisnotasum}
Let $G$ be  a reduced conical sandpile graph.  Let $\mathbb{F}_{G}$ denote the free abelian  monoid on the set $G^0 $.    Suppose $v\in G^0$ and suppose there exists a nonzero element  $ x \in \mathbb{F}_G$ for which $x\neq v$  in $\mathbb{F}_G$,   but $x = v$ in $\SP(G)$.   Then there exists   a vertex $u_v\in G^0 \setminus \{s\} $ for which 
$$s^{-1}(u_v) = \{ e_v, e_{1}, \dots , e_{n_{u_v}-1}\},$$
where $n_{u_v}$ denotes $|s^{-1}_G(u_v)|$, $r(e_v) = v$ and $r(e_{i}) = s$ for all $1\leq i \leq n_{u_v} - 1$.    (Note that we allow $u_v=v$, and $e_v$ to be a  loop at $v$.)  
\end{lemma}

\begin{proof}
 

 If $x=v$ in $\SP(G)$ but $x\neq v$ in $\mathbb{F}_{G^0 \setminus \{s\}}$, then by the definition of $\sim$ (see Display (\ref{hfgtrgt6551})) 
   there exists a sequence of positive length 
  $x = a_0 \rightsquigarrow a_1 \rightsquigarrow \dots \rightsquigarrow a_{i-1}  \rightsquigarrow a_i \rightsquigarrow \dots \rightsquigarrow  a_n = v$ where  $a_{i-1}  \rightsquigarrow a_i$ denotes either $\rightarrow$ or $\leftarrow$.    
     We focus on the last move in the sequence, $a_{n-1} \rightsquigarrow  a_n = v$.   Because $G$ is reduced, each non-sink vertex of $G$ emits at least two edges, so the $\rightsquigarrow$ connecting $a_{n-1}$ to $v$ cannot be $\leftarrow$.  
      So we must have $a_{n-1} \rightarrow v$.  That is,  for some $u_v \in G^0 \setminus \{s\}$ 
      we have that $n_{u_v}u_v$ is in ${\rm supp}(a_{n-1})$, and 
  $$ ( a_{n-1} - n_{u_v}u_v ) + \sum_{\{e\in E^1 | s(e) = u_v, r(e) \neq s \} }r(e)  = v. $$
 Since $v$ is an atom in  $\mathbb{F}_{G}$, there are two possibilities:
  
  \underline{Case 1}:  $a_{n-1}  - n_{u_v}u_v = v$, and  $ \sum_{\{e\in E^1 | s(e) = u_v, r(e) \neq s \} }r(e) = 0$ ;  or
  
  \underline{Case 2}:  $a_{n-1} - n_{u_v}u_v = 0$, and  $ \sum_{\{e\in E^1 | s(e) = u_v, r(e) \neq s \} }r(e) = v$. 
  
  The equation $ \sum_{\{e\in E^1 | s(e) = u_v, r(e) \neq s \} }r(e) = 0$ of Case 1 would give that all edges $e$ for which $s(e) = u_v$ have $r(e) = s$;  since $n_{u_v} \geq 2$ this would give $n_{u_v}u_v = 0$ in $\SP(G)$, which would violate that $G$ is conical.  So Case 1 cannot occur.  
  
  But the equations given  in Case 2 describe exactly the desired configuration of vertices and edges in $G$. \end{proof}
  
 



\begin{proposition}\label{refinement}
Let $G$ be a  reduced conical sandpile graph.  Suppose in addition that $\SP(G)$ is a refinement monoid.     Then there is a partition $\mathcal{P}$ of $G^0 \setminus \{s\}$ for which the graph $G$ has the following form.

Let $P = \{v_1, v_2, \dots , v_t\} $ be an equivalence class in $ \mathcal{P}$.     Then each $v_i$ emits all but one of its edges to the sink $s$, and  this one not-to-the-sink edge $e_i$ has $s(e_i) = v_i $ for each $1 \leq i \leq t$, $r(e_i) = v_{i-1}$ for each $2 \leq i \leq t$, and $r(e_1) = v_t$.  

Denote by $G_P$ the subgraph of $G$ having  $G_P^0 = \{s, v_1, \dots , v_t\}$ and $G_P^1 = \cup_{v_i \in P}s^{-1}(v_i)$.      For equivalence classes $P \neq Q \in \mathcal{P}$, $G_P^0 \cap G_Q^0 = \{s\}$ and $G_P^1 \cap G_Q^1 = \emptyset.$

Then $G$ is the   union of  the subgraphs $\{G_P \ | \ P \in \mathcal{P}\}$. 
\end{proposition}

\begin{proof}   By Lemma \ref{atomcancel}, the monoid $\SP(G)$ contains no atoms.  In particular each vertex $v\in G^0 \setminus \{s\}$ can be written as the sum of two nonzero elements in $\SP(G)$.  Since $v$ is an atom in $\mathbb{F}_G$ we have that the hypotheses of   Lemma \ref{vertexisnotasum} are satisfied.  So for each $v\in G^0 \setminus \{s\}$ there exists a vertex $u_v$ with the indicated configuration.   By considering this  configuration,  if a vertex $u \in G^0 \setminus \{s\}$ arises as $u_v$ for some vertex $v$, then $v$ is unique. (In other words, if $u \in G^0 \setminus \{s\}$ has $u = u_v = u_{v'}$ for vertices $v, v'$ in $G^0 \setminus \{s\} $, then $v = v'$.)  
 Since such a configuration must occur  for every $v\in G^0 \setminus \{s\}$, by the Pigeon Hole Principle applied to $G^0 \setminus \{s\}$ (recall that by definition any sandpile graph is finite)  we have that each $u\in G^0 \setminus \{s\}$ appears exactly once as $u_v$ for some (unique) $v\in G^0 \setminus \{s\}$.  Now partition $G^0 \setminus \{s\}$ using the transitive closure of this relationship.    
 \end{proof}

\begin{corollary}\label{SP(G)isdirectsumCsubn}
Suppose $G$ is a  conical  sandpile graph for which $\SP(G)$ is a refinement monoid.  Then  
\[\SP(G) \cong  \bigoplus_{i=1}^t C_{n_i},\]
 for some  integers $n_1, \dots , n_t \geq 2$.  Conversely, any monoid of the form $ \bigoplus_{i=1}^t C_{n_i}$ with $n_1, \dots , n_t \geq 2$ is a conical sandpile monoid which is also  a refinement monoid.   
\end{corollary}

\begin{proof}  This follows directly from the observations made in   Example \ref{weightedcycle}, together with Remark \ref{sandpilemonoidremark} and Proposition \ref{refinement}.  
\end{proof}

\begin{proposition}\label{V(L(E))issandpile}
Let $E$ be a finite graph.   Let $L_\K(E)$ denote the (unweighted) Leavitt path algebra.  Then the monoid $\mathcal{V}(L_\K(E))$ is isomorphic to the   sandpile monoid $\SP(G)$ of a  sandpile graph $G$ if and only if $\mathcal{V}(L_\K(E)) \cong \oplus_{i=1}^t C_{n_i}$ for some positive integers $n_1, \dots , n_t$.  
\end{proposition}

\begin{proof}   Sufficiency of the statement follows from Corollary  \ref{SP(G)isdirectsumCsubn}.  

 For necessity, we note that   the $\mathcal{V}$-monoid of any ring is conical, while the $\mathcal{V}$-monoid of the (unweighted)  Leavitt path algebra $L_\K(E)$ of any graph $E$ is   refinement (see e.g. \cite[Theorem 3.6.8]{TheBook}).  So if $\mathcal{V}(L_\K(E)) \cong \SP(G)$ for some  sandpile graph $G$ then  Corollary \ref{SP(G)isdirectsumCsubn} gives the result.   
 \end{proof}



\begin{remark}\label{whatisL(E)whenV(L(E))issandpile}   By Proposition \ref{V(L(E))issandpile} we know the structure of the $\mathcal{V}$-monoid of any unweighted Leavitt path algebra whose $\mathcal{V}$-monoid is a finite sandpile monoid. 
We discuss  the specific case where  $\mathcal{V}(L_\K(E)) \cong C_n$ (i.e., that there is only one term in the direct sum given in Proposition  \ref{V(L(E))issandpile}).    By \cite[Theorem 3.1.10 or Lemma 6.3.14]{TheBook},  $\mathcal{V}(L_\K(E)) \cong C_n$ yields that $L_\K(E)$ is a purely infinite simple algebra.  
Moreover, $K_0(L_\K(E)) \cong \mathbb{Z}_{n-1}$.      Fix any isomorphism $\varphi: K_0(L_\K(E)) \to \mathbb{Z}_{n-1}$, and let $t$ denote $ \varphi([L_\K(E)])$.     

Let $R_n$ denote the graph having one vertex and $n$ loops at that vertex.  It is well known that  $K_0({\rm M}_\ell(L_\K(R_n))) \cong  \mathbb{Z}_{n-1}$, for each $\ell \in \mathbb{N}^+$, and that under this isomorphism $[{\rm M}_\ell(L_\K(R_n))] \mapsto \ell \in \mathbb{Z}_{n-1}$.  
 Therefore $K_0(L_\K(E)) \cong K_0({\rm M}_t(L_\K(R_n)))$, via a map that takes $[L_\K(E)]$ to $[{\rm M}_t(L_\K(R_n))]$.  Let $R_{n,t}$ denote the standard graph having $L_\K(R_{n,t}) \cong {\rm M}_t(L_\K(R_n))$.  Indeed, $R_{n,t}$  is the graph $R_n$ with a path of length $t-1$ attached to the vertex;  this is exactly the construction described in the proof of Theorem \ref{kspcts}.      It is well-known that ${\rm det}(I - A_{R_{n,t}}) $ is negative. 
 
If $E$ has the property that ${\rm det}(I - A_E) <0$, then  by the Classification Theorem of purely infinite simple Leavitt path algebras \cite[Theorem 6.3.32]{TheBook}, we  obtain that $L_\K(E)$ is isomorphic to   $  M_t(L_\K(R_n))$.   

On the other hand, if $E$ has the property that ${\rm det}(I - A_E) > 0$, then as of the writing of this article we do not know whether $L_\K(E)$ is necessarily isomorphic to $  {\rm M}_t(L_\K(R_n))$.  We note that  this barrier  lies at the heart of the longstanding {\it Algebraic Kirchberg Phillips Question for Leavitt Path Algebras of Finite Graphs}, see e.g. \cite[Question 6.3.3]{TheBook}.    (Note: the situation ${\rm det}(I - A_E) =0$ cannot occur for such $E$.)  



\end{remark}

We conclude this article by presenting the following natural definition.   An investigation into various properties of these structures will be taken up in forthcoming work.  

\begin{deff}
Let $\K$ be a field.  We call a $\K$-algebra $A$ a {\it sandpile  $\K$-algebra} in case there exists a conical sandpile graph $E$ for which 
$$A \cong L_\K(E/S, w_r),$$
where $S$ denotes the set of vertices of $E$ which do not connect to a cycle, and $w_r$ is the restriction of the balanced weighting on $E$ to $E/S$.   
\end{deff}

In particular if $E$ is a sandpile graph such that each non-sink vertex connects to a cycle, then considering $E$ as a balanced weighted graph,  $L_\K(E / \{s\},w_r)$ is a sandpile $\K$-algebra.      Among other things, the collection of  sandpile $\K$-algebras provides an umbrella (significantly  smaller than the umbrella provided by the collection of weighted Leavitt path $\K$-algebras) under which {\it all} of Leavitt's $\K$-algebras $L_\K(n,n+k)$ (for any pair $n,k \in \mathbb{N}^+$) may be realised.

\section{Acknowledgements}
A part of this work was done when the second author was an Alexander von Humboldt Fellow at the University of M\"unster in the Winter of 2021. He would like to thank both institutions for an excellent hospitality.





\end{document}